\pdfoutput=1
\documentclass[11pt,a4paper]{amsart}
\usepackage[left=25truemm,right=25truemm,top=30truemm,bottom=30truemm]{geometry}
\usepackage{amsmath, amsfonts, amssymb, amsthm, amscd, latexsym, mathtools}
\usepackage{ascmac}
\usepackage{enumerate}
\usepackage{amsthm}
\usepackage{xcolor}
\usepackage{hyperref}
\usepackage{cleveref}
\usepackage{mathrsfs}
\usepackage{enumitem}

\numberwithin{equation}{section}

\newtheorem{Thm}{Theorem}[section]
\newtheorem{Lem}[Thm]{Lemma}
\newtheorem{Prop}[Thm]{Proposition}
\newtheorem{Cor}[Thm]{Corollary}
\theoremstyle{definition}
\newtheorem{Def}[Thm]{Definition}
\newtheorem{Rem}[Thm]{Remark}
\newtheorem{Ex}[Thm]{Example}
\newtheorem{problem}{Problem}

\newcommand{\RR}{\mathbb R}

\newcommand{\ZZ}{\mathbb Z}

\newcommand{\GRHn}{G_{\mathbb{R} \mathrm{H}^n}}
\newcommand{\geRHn}{\mathfrak{g}_{\mathbb{R} \mathrm{H}^n}}
\newcommand{\RRH}{\mathbb{R} \mathrm{H}}
\newcommand{\LR}{\langle, \rangle}
\newcommand{\gee}{\mathfrak{g}}
\newcommand{\haa}{\mathfrak{h}}
\newcommand{\support}{\mathrm{supp}}
\newcommand{\Modin}{\mathop{\widetilde{\mathfrak{M}}}}

\newcommand{\RRtimesaut}{\mathbb{R}_{>0} \times \mathrm{Aut}}
\newcommand{\Lstat}{\mathop{\mathcal{L}\mathrm{Stat}}}
\newcommand{\LstatDF}{\mathcal{L}\mathrm{Stat}^{\mathrm{DF}}}
\newcommand{\LstatCS}{\mathcal{L}\mathrm{Stat}^{\mathrm{CS}}}

\newcommand{\MLStat}{\mathop{\mathcal{ML}\mathrm{Stat}}}
\newcommand{\MLStatDF}{\mathop{\mathcal{ML}\mathrm{Stat}^{\mathrm{DF}}}}
\newcommand{\MLStatCS}{\mathop{\mathcal{ML}\mathrm{Stat}^{\mathrm{CS}}}}


\title[The moduli spaces of left-invariant statistical structures on Lie groups]{The moduli spaces of left-invariant statistical structures on Lie groups}
\author{Hikozo Kobayashi, Yu Ohno, Takayuki Okuda, Hiroshi Tamaru}
    \thanks{
    The first author was supported by JST SPRING, Grant Number JPMJSP2132.
    The second author was supported by JST SPRING, Grant Number JPMJSP2119.
    The third author was supported by JSPS KAKENHI, Grant Number JP24K06714.
    The fourth author was supported by JSPS KAKENHI Grant Numbers JP23K22395 and JP24K21193.
    This work was partially supported by the MEXT Promotion of Distinctive Joint Research Center Program JPMXP0723833165 and the Osaka Metropolitan University Strategic Research Promotion Project (Development of International Research Hubs).
    This work was also supported by the Research Institute for Mathematical Sciences, an International Joint Usage/Research Center at Kyoto University.
    }
\address[H.~Kobayashi]{%
    Mathematics Program, Graduate School of Advanced Science and Engineering, Hiroshima University, 
    1-3-1 Kagamiyama, Higashi-Hiroshima City, Hiroshima, 739-8526, Japan
        }
\email{hikozo-kobayashi@hiroshima-u.ac.jp}
\address[Y.~Ohno]{%
    Department of Mathematics, Faculty of Science, Hokkaido University,
    Kita 10, Nishi 8, Kita-Ku, Sapporo, Hokkaido, 060-0810, Japan
        }
\email{ono.yu.0414@gmail.com}
\address[T.~Okuda]{%
    Mathematics Program, Graduate School of Advanced Science and Engineering, Hiroshima University, 
    1-3-1 Kagamiyama, Higashi-Hiroshima City, Hiroshima, 739-8526, Japan
        }
\email{okudatak@hiroshima-u.ac.jp}
\address[H.~Tamaru]{%
    Department of Mathematics, Graduate School of Science, Osaka Metropolitan University,
    3-3-138, Sugimoto, Sumiyoshi-ku, Osaka, 558-8585, Japan
        }
\email{tamaru@omu.ac.jp}
\begin{document}

\begin{abstract}
In the context of information geometry, the concept known as left-invariant statistical structure on Lie groups is defined by Furuhata--Inoguchi--Kobayashi (Inf Geom 4(1):177--188, 2021).
In this paper, we introduce the notion of the moduli space of left-invariant statistical structures on a Lie group. We study the moduli spaces for three particular Lie groups, each of which has a moduli space of left-invariant Riemannian metrics that is a singleton. 
As applications, we classify left-invariant conjugate symmetric statistical structures and left-invariant dually flat structures (which are equivalent to left-invariant Hessian structures) on these three Lie groups. 
A characterization of the Amari--Chentsov $\alpha$-connections on the Takano Gaussian space is also given.
\end{abstract}

\maketitle

\tableofcontents

\section{Introduction}\label{sec1}

In differential geometry, 
classifying geometric structures on manifolds is a fundamental problem.
When the manifolds are equipped with additional geometric structures or suitable group actions,
it becomes natural to consider compatible geometric structures.
A particular area of interest is left-invariant geometric structures on Lie groups,
which have emerged as a significant field of study.
It is crucial to examine whether given Lie groups admit particular left-invariant geometric structures and,
furthermore, to classify those structures on the given Lie groups.
In this context, the concept of moduli spaces of left-invariant geometric structures plays a vital role, providing insight into their classification and deformation theory (e.g.,~\cite{ABD_2011, Luis-T, HTT, KTT, KONDO, KOTT, L-DL, Salamon_2001}).

In this paper, we focus on \emph{statistical structures} on manifolds (cf.~\cite{Kurose-1990, L-SM}). 
A statistical structure originates from information geometry, and is defined as a pair consisting of a Riemannian metric and a torsion-free affine connection that satisfies suitable conditions (see Section~\ref{Sec_SM} for details). In the theory of statistical structures, two distinguished subclasses have received particular attention. The first is the class of \emph{dually flat structures}, defined by the property that the associated affine connection is flat (cf.~\cite{AN}). Dually flat structures play an important role in information geometry (see~\cite{AN} for details). The second subclass is the class of \emph{conjugate symmetric statistical structures} (cf.~\cite{L-SM}), which is a broader concept that includes dually flat structures. This structure has significant connections with affine hypersurface theory (see~\cite{Opozda_2019}) and constant curvature statistical manifolds (see~\cite{KO-CC}).
It is an important problem to examine whether given manifolds admit these structures and to classify such structures on those manifolds.

A statistical structure on a Lie group is said to be 
\emph{left-invariant} 
if both the metric and the connection are left-invariant (cf.~\cite{FIK, IO-2024}). 
In this paper, we focus on this type of structure, and our concern is the following problem:

\begin{problem}\label{Prob_A}
Given a Lie group $G$, find all left-invariant dually flat structures and left-invariant conjugate symmetric statistical structures on $G$.
\end{problem}

Here, we review some previous works related to Problem~\ref{Prob_A}.
We note that a dually flat structure is equivalent to a Hessian structure (see~\cite{Shima}).
Shima~\cite{Shima-1980} deeply studied this structure, 
and proved that left-invariant Hessian structures are admissible only on solvable Lie groups.
Note that not every solvable Lie group admits a left-invariant Hessian structure, 
and a complete classification is still far from being achieved. 
For left-invariant conjugate symmetric statistical structures on Lie groups, 
a typical example is provided by the space $\mathcal{N}$ of univariate normal distributions, 
which can be identified with the Lie group $\RR_{>0} \ltimes \RR$. 
Furuhata--Inoguchi--Kobayashi~\cite{FIK} showed that the 
Amari--Chentsov $\alpha$-connections on the space $\mathcal{N}$ 
are the only left-invariant affine connections that are conjugate symmetric
with respect to the Fisher metric on $\mathcal{N}$. 
Furthermore, Kobayashi and the second author \cite{KO} proved that a similar statement holds for the Lie group of multivariate normal distributions. 
Recently, Inoguchi and the second author \cite{IO-2024} classified
left-invariant conjugate symmetric statistical structures on three-dimensional Lie groups. 
Moreover, for the classification of a certain class of left-invariant statistical structures on two- and three-dimensional Lie groups, see also \cite{MNF_2024}.
For higher-dimensional Lie groups, aside from the Lie group of multivariate normal distributions studied in \cite{KO},
the classification problem becomes more complicated and remains widely open. 

In this paper, we introduce a new framework for investigating Problem~\ref{Prob_A}, 
which can be applied to Lie groups of any dimension. 
In the study of left-invariant Riemannian metrics on Lie groups,
the notion of moduli spaces plays an important role. 
This concept was introduced for the Iwasawa manifold in \cite{Scala_2013} and further developed by Kodama, Takahara, and the fourth author in \cite{KTT}.
Consider the space of all left-invariant Riemannian metrics on a Lie group $G$, and the natural action of $\RR_{>0} \times \mathrm{Aut}(G)$ on it. 
The moduli space $\mathfrak{PM}(G)$ of left-invariant Riemannian metrics on $G$ is defined as the orbit space of this action. 
Inspired by this concept, let us denote by $\Lstat(G)$ the space of all left-invariant statistical structures on $G$, and consider the natural group action on $\Lstat(G)$ by $\RR_{>0} \times \mathrm{Aut}(G)$. 
We define the \emph{moduli space of left-invariant statistical structures on $G$} as the orbit space of this action, 
and denote it by $\MLStat(G)$.

We will see that 
two left-invariant statistical structures $(g,\nabla)$ and $(g',\nabla')$ lie in the same $\RR_{>0} \times \mathrm{Aut}(G)$-orbit
 if and only if 
there exists a \emph{scaling statistical isomorphism compatible with the group structure} (see Definition~\ref{def:ssi_compati-grp}) between statistical Lie groups $(G,g,\nabla)$ and $(G,g',\nabla')$.
In particular, the action of $\RR_{>0} \times \mathrm{Aut}(G)$ on $\Lstat(G)$ preserves the properties of statistical structures,
such as the dually flat property and conjugate symmetry.
Consequently, we define two subspaces of $\MLStat(G)$, denoted by $\MLStatCS(G)$ and $\MLStatDF(G)$, 
as the \emph{moduli spaces of left-invariant conjugate symmetric statistical structures}
and \emph{left-invariant dually flat structures}, respectively. 
In this paper, motivated by Problem~\ref{Prob_A}, we address the following problem:

\begin{problem}\label{Prob_B}
Given a Lie group $G$, determine the moduli spaces $\MLStatCS(G)$ and $\MLStatDF(G)$. 
\end{problem}

For example, if $\MLStatDF(G) = \emptyset$, 
then the Lie group $G$ does not admit any left-invariant dually flat structures.
If $\MLStatDF(G)$ is a singleton, then the uniqueness of such a structure holds. 
We note that $\MLStatCS(G) \neq \emptyset$ always holds, 
since the pair of a left-invariant Riemannian metric and its Levi-Civita connection is conjugate symmetric.
Therefore, regarding $\MLStatCS(G)$, 
we are interested in the existence of nontrivial left-invariant conjugate symmetric statistical structures.

It is natural to first consider Problem~\ref{Prob_B} for Lie groups $G$ where the moduli space $\MLStat(G)$ is small. 
Recall that a left-invariant statistical structure on a Lie group $G$ is defined as a pair of a left-invariant Riemannian metric and a left-invariant affine connection.
We are thus led to consider Lie groups $G$ whose moduli space $\mathfrak{PM}(G)$ of left-invariant Riemannian metrics is small, since these provide natural candidates for the study of Problem~\ref{Prob_B}.
In fact, the simplest case has been classified.
Specifically, for a connected and simply-connected Lie group, 
the moduli space $\mathfrak{PM}(G)$ is a singleton 
(that is, a left-invariant Riemannian metric is essentially unique) 
if and only if 
it is isomorphic to one of the following three series (cf.~\cite{L-DL, KTT}): 
\begin{equation}\label{eq_special-Liegrps_Intro}
    \RR^n,\;
    G_{\RR \mathrm{H}^n}\ (n \ge 2),\;
    H^3 \times \RR^{n-3}\ (n \ge 3).
\end{equation}
Note that $\RR^n$ is an abelian Lie group, 
$\GRHn$ is the Lie group associated with the real hyperbolic space $\RR \mathrm{H}^n$ 
(the solvable part of the Iwasawa decomposition of the identity component $SO_0(n,1)$ of $SO(n,1)$, and acts simply-transitively on $\RR \mathrm{H}^n$), 
and $H^3$ is the three-dimensional Heisenberg group. 
It is well-known that their unique metrics are flat on $\RR^n$, negative constant sectional curvature on $\GRHn$,
and Ricci soliton on $H^3 \times \RR^{n-3}$, respectively. 

The main result of this paper addresses Problem~\ref{Prob_B} 
for the three series of Lie groups shown in \eqref{eq_special-Liegrps_Intro}. 
For each of these Lie groups $G$, Table~\ref{table_Main-1} summarizes the sets of representatives for the moduli spaces $\MLStatCS(G)$ and $\MLStatDF(G)$. 

For these Lie groups $G$, 
let us denote by $S^3(\mathfrak{g}^\ast)$ the space of all symmetric $(0,3)$-tensors on the Lie algebra $\mathfrak{g}$ of $G$.
Then the moduli space $\MLStat(G)$ 
can be written as an orbit space of $S^3(\mathfrak{g}^\ast)$ under a certain Lie group (see Section~\ref{subsection:modli_unique_metric}).  
A set of representatives for the moduli spaces $\MLStatCS(G)$, resp.~$\MLStatDF(G)$, 
refers to a subset of $S^3(\mathfrak{g}^\ast)$ 
whose projection onto these moduli spaces is surjective. 
If this projection is bijective, we call it a complete set of representatives.
The notations $V^+$, $v_0$, and $w_0$ appearing in Table~\ref{table_Main-1} will be explained in Theorems~\ref{thm:Rn-intro}, \ref{thm:GRHn-intro}, and~\ref{thm:H3Rn-3-intro}, respectively. 

\begin{table}[h]
    \centering
    \caption{Sets of representatives for $\MLStatCS(G)$ and $\MLStatDF(G)$ in \eqref{eq_special-Liegrps_Intro}}
    \label{table_Main-1}
    \begin{tabular}{c|cc}
        $G$ & $\MLStatCS(G)$ & $\MLStatDF(G)$ \\ \hline
        $\RR^n$ & $S^3(\mathfrak{g}^\ast_{\RR^n})$ & $V^+$\\
        $G_{\RR \mathrm{H}^n}$ & $\RR v_0$ & $\{v_0, -v_0\}$ \\
        $H^3$ & $\{0\}$ & $\emptyset$ \\
        $H^3\times \RR^{n-3}$ & $(\mathbb{R} w_0 \odot S^1(\mathfrak{g}_{\mathbb{R}^{n-3}}^\ast)) \oplus S^3(\mathfrak{g}_{\mathbb{R}^{n-3}}^\ast)$ & $\emptyset$
    \end{tabular}
\end{table}

In order to express the moduli spaces, we need to fix the basis of the Lie algebras. 
For each of our three Lie algebras $\gee_{\RR^n}$, $\geRHn$ and $\haa_3 \oplus \gee_{\RR^{n-3}}$, the standard basis 
$\{ e_1, e_2, \ldots , e_n \}$ 
is a basis whose nonzero bracket relations are given as follows: 
\begin{align*}
\gee_{\RR^n} &: \mbox{(none)} , \\
\geRHn &: [e_1, e_i] = e_i \quad (i=2, \ldots , n) , \\ 
\haa_3 \oplus \gee_{\RR^{n-3}} &: [e_1,e_2] = e_3 . 
\end{align*}
Let $\{ x_1 , \ldots , x_n \} \subset \mathfrak{g}^\ast$ be the dual basis of the standard basis. 
We will now describe the results for each of the three Lie groups in detail.
The first case is the abelian Lie group $\RR^n$. 

\begin{Thm}[see also Theorem~\ref{thm:Rn-moduli}]\label{thm:Rn-intro}
For the abelian Lie group $\RR^n$, the following holds:
\begin{enumerate}[label=(\arabic*)]
\item
All left-invariant statistical structures on $\RR^n$ are conjugate symmetric. 
\item 
The following is a set of representatives for $\MLStatDF(\RR^n)$$:$
\[
    V^+ := \left\{ \sum_{i = 1}^n \lambda_i x_i^3 ~\middle|~ \lambda_1 \geq \dots \geq \lambda_n \geq 0 \right\}.
\]
The moduli space $\MLStatDF(\RR^n)$ is homeomorphic to the orbit space of the action of $\RR_{>0}$ on $V^+$, which implies the existence of an $(n-1)$-dimensional family of left-invariant dually flat structures on $\RR^n$.
\end{enumerate}
\end{Thm}

The second case is the Lie group $\GRHn$ of the real hyperbolic space. 
\begin{Thm}[see also Theorem~\ref{thm:GRHn-cs-df-moduli}]\label{thm:GRHn-intro}
For the Lie group $\GRHn$, we denote by 
\[
\textstyle 
v_0 := 4 x_1^3 + 6 \sum_{i=2}^{n} x_1 x_i^2 \in S^3(\geRHn^\ast) .
\] 
\begin{enumerate}[label=(\arabic*)]
\item
$\RR v_0$ is a complete set of representatives for $\MLStatCS(\GRHn)$. 
Specifically, the moduli space, equipped with the natural quotient topology, is homeomorphic to 
the subspace $\RR v_0$ of $S^3(\geRHn^\ast)$. 
This implies that
there exists a one-dimensional family of left-invariant conjugate symmetric statistical structures on $\GRHn$.
\item 
$\{ v_0, -v_0 \}$ is a complete set of representatives for $\MLStatDF(\GRHn)$.
Specifically, the moduli space, equipped with the natural quotient topology, is homeomorphic to a two-point discrete space. 
This implies that there exist exactly two left-invariant dually flat structures on $\GRHn$ up to scaling statistical isomorphisms compatible with the group structure.
\end{enumerate}
\end{Thm}

The third case is the Lie group $ H^3 \times \mathbb{R}^{n-3} $,
the direct product of the three-dimensional Heisenberg group and an abelian Lie group. 
We denote by $\odot$ the symmetric tensor product. 

\begin{Thm}[see also Theorem~\ref{thm:H3Rn-3-cs-df-moduli}]\label{thm:H3Rn-3-intro}
For the Lie group $ H^3 \times \mathbb{R}^{n-3} $, we denote by 
$$
w_0 := x_1 x_1 + x_2 x_2 + x_3 x_3 \in S^2(\mathfrak{h}_3^\ast) . 
$$
\begin{enumerate}[label=(\arabic*)]
\item
$(\mathbb{R} w_0 \odot S^1(\mathfrak{g}_{\mathbb{R}^{n-3}}^\ast)) \oplus S^3(\mathfrak{g}_{\mathbb{R}^{n-3}}^\ast) $
is a set of representatives for $ \MLStatCS(H^3 \times \mathbb{R}^{n-3}) $.
In particular, if $ n=3 $, then the set of representatives for $ \MLStatCS(H^3) $ is a singleton, 
which implies that a left-invariant conjugate symmetric statistical structure on $ H^3 $ is essentially unique;
this structure consists of a left-invariant Riemannian metric and its Levi-Civita connection. 
If $n \geq 4$, then the moduli space  
$\MLStatCS(H^3 \times \mathbb{R}^{n-3})$ 
is homeomorphic to the orbit space of $(\mathbb{R} w_0 \odot S^1(\mathfrak{g}_{\mathbb{R}^{n-3}}^\ast)) \oplus S^3(\mathfrak{g}_{\mathbb{R}^{n-3}}^\ast)$ under the action of $O(n-3)$.
\item
There do not exist any left-invariant dually flat structures on $ H^3 \times \mathbb{R}^{n-3} $. 
\end{enumerate}
\end{Thm}

We note that the statement regarding $ H^3 $ has been established in \cite{IO-2024}. 
Our argument provides an alternative approach and generalizes this result to the higher-dimensional case. 
In particular, unlike the case of $n=3$, 
the Lie group $H^3 \times \mathbb{R}^{n-3}$ with $n \geq 4$ admits 
nontrivial left-invariant conjugate symmetric statistical structures. 

Finally, in this section, we present one application of our theorems. 
By using an explicit expression for the moduli space $\MLStatCS(\GRHn)$ 
of left-invariant conjugate symmetric statistical structures on $\GRHn$, 
we obtained a characterization of the Amari--Chentsov $\alpha$-connections on the Takano Gaussian space. 
This result is analogous to the characterization of the Amari--Chentsov $\alpha$-connections
on the family of univariate normal distributions (cf.~\cite{FIK}), 
and on the family of multivariate normal distributions (cf.~\cite{KO}).

\begin{Cor}[see also Theorem~\ref{Thm_Amari--Chentsov_Takano}]
The Amari--Chentsov $\alpha$-connection $\nabla$ on the Takano Gaussian space $(\mathcal{N}_T, g^F)$ is 
characterized by the following two properties: 
\begin{enumerate}[label=(\arabic*)]
    \item The statistical structure $(g^F, \nabla)$ is conjugate symmetric. 
    \item $\nabla$ is $\RR_{>0} \ltimes \RR^n$-invariant. 
\end{enumerate}
\end{Cor} 

This paper is organized as follows:
In Section~\ref{Sec_pre}, we briefly recall the basic facts about statistical manifolds and introduce several important classes and examples of statistical structures.
We begin in Section~\ref{Sec_Lie} with a study of left-invariant statistical structures on Lie groups.
Section~\ref{subsec:remark_computations} contains several computational formulas that will be used later, in Section~\ref{Sec_Rn} and the subsequent sections.
The notion of the moduli space of left-invariant statistical structures is then introduced in Section~\ref{Sec_LISS}.
Finally, Sections~\ref{Sec_Rn}, \ref{Sec_GRHn}, and \ref{Sec_H3Rn-3} present classification results together with the corresponding moduli spaces for the Lie groups listed in \eqref{eq_special-Liegrps_Intro}.

\section{Preliminaries}\label{Sec_pre}
Throughout this section, we fix a smooth manifold $M$.

\subsection{Notations}
In this subsection, we fix the terminology used throughout this paper.

For each affine connection $\nabla$ on $M$, we use the symbols $R^{\nabla}$ and $T^{\nabla}$ for the curvature tensor field of type $(1,3)$ and the torsion tensor field of type $(1,2)$ of $\nabla$, respectively,
that is, 
\begin{align*}
    R^\nabla(X,Y)Z &:= (\nabla_X \nabla_Y - \nabla_Y \nabla_X - \nabla_{[X,Y]})Z, \\
    T^\nabla(X,Y) &:= \nabla_X Y - \nabla_Y X - [X,Y].
\end{align*}

Let us fix a Riemannian metric $g$ and an affine connection $\nabla$ on $M$. 
We write $\nabla^g$ for the Levi-Civita connection of $g$ on $M$.
The dual connection of $\nabla$ with respect to $g$ will be denoted by $\overline{\nabla}$, 
that is, 
\[
    Xg(Y, Z) = g(\nabla_X Y, Z) + g(Y, \overline{\nabla}_X Z).
\]
The following relation between $\nabla$ and $\overline{\nabla}$ is well-known.

\begin{Prop}[\cite{Matsuzoe_2010}]\label{Prop_dual-connections}
If two of the following four conditions are satisfied,
then the remaining two will follow:
\begin{enumerate}
    \item[(1)] $T^{\nabla} \equiv 0$,
    \item[(2)] $T^{\overline{\nabla}} \equiv 0$,
    \item[(3)] $(\nabla + \overline{\nabla})/2 = \nabla^g$,
    \item[(4)] The $(0,3)$-tensor field $\nabla g$ is totally symmetric.
\end{enumerate}
\end{Prop}

The following formula is also well-known (see \cite{Nomizu-Sasaki}):
\begin{equation}\label{eq_curv-dual}
    g(R^{\nabla}(X,Y)Z ,W) + g(R^{\overline{\nabla}}(X,Y)W,Z) = 0.
\end{equation}
In particular, we have
\begin{equation}\label{eq_curv-dual-2}
    R^{\nabla} \equiv 0 \iff 
    R^{\overline{\nabla}} \equiv 0.
\end{equation}

\subsection{Statistical manifolds}\label{Sec_SM}
In this subsection, we recall some basic facts about statistical manifolds, see \cite{AN, L-SM} for details. 

In this paper, we say that 
a triple $(M,g,\nabla)$ is called a \emph{statistical manifold} 
if $g$ is a Riemannian metric and $\nabla$ is an affine connection on $M$ satisfying all conditions in Proposition~\ref{Prop_dual-connections}.
For a statistical manifold $(M,g,\nabla)$, the pair $(g,\nabla)$ is called a \emph{statistical structure} on $M$.
In particular, $\nabla$ is called a \emph{statistical connection} on the Riemannian manifold $(M,g)$.

For each Riemannian metric $g$ on $M$, 
the pair $(g, \nabla^g)$ is a statistical structure on $M$ by the definition of the Levi-Civita connection, and $(g, \nabla^g)$ is called the \emph{trivial statistical structure} (cf.~\cite{Furuhata_2009}).
Thus, statistical manifolds can be considered as a generalization of Riemannian manifolds. 

For each statistical structure $(g,\nabla)$, the pair $(g, \overline{\nabla})$ is also a statistical structure on $M$.
The pair $(g, \overline{\nabla})$ is called the \emph{dual statistical structure} of the statistical structure $(g, \nabla)$, and the triple $(M, g, \overline{\nabla})$ is called the \emph{dual statistical manifold} of $(M, g, \nabla)$.

We define an equivalence relation between statistical manifolds as follows.

\begin{Def}\label{Def_SIUS}
Let $(M,g, \nabla)$ and $(M', g', \nabla')$ be statistical manifolds. 
If there exists a diffeomorphism
$f \vcentcolon M \to M'$ 
and a positive real number $r > 0$ satisfying the following two conditions (1) and (2), 
then $(M,g, \nabla)$ and $(M', g', \nabla')$ are said to be \emph{statistically isomorphic up to scaling by $(f,r)$}:
\begin{enumerate}[label=(\arabic*)]
\item $f^{\ast} \nabla' = \nabla$,
\item $f^{\ast} g' = r \cdot g$.
\end{enumerate}
The pair $(f, r)$ is called a \emph{scaling statistical isomorphism} between statistical manifolds $(M,g,\nabla)$ and $(M',g',\nabla')$.
In particular, in the case $r = 1$, the map $f$ is called a \emph{statistical isomorphism} from $(M, g, \nabla)$ to $(M', g', \nabla')$, and we say that $(M, g, \nabla)$ and $(M', g', \nabla')$ are \emph{statistically isomorphic} to each other.
\end{Def}

The following proposition will be applied in Theorem~\ref{Thm_DF-CS-preser}.

\begin{Prop}\label{Prop_stat-isom-uptos}
Let $(M,g, \nabla)$ and $(M', g', \nabla')$ be statistical manifolds that are statistically isomorphic up to scaling by $(f, r)$.
Then the following equality holds:
\[
f^\ast \overline{\nabla'} = \overline{\nabla}.
\]
\end{Prop}

Next, we introduce the \emph{difference tensor} and the \emph{cubic form}, which provide equivalent descriptions of statistical connections on Riemannian manifolds.
Let $(g, \nabla)$ be a statistical structure on $M$. 
We define the $(1,2)$-tensor field $K$ by $K(X, Y) = \nabla_X Y - \nabla^g_X Y$.
This tensor field $K$ is called the \emph{difference tensor} of $(g,\nabla)$.
Since the affine connections $\nabla$ and $\nabla^g$ are both torsion-free, the difference tensor $K$ is symmetric, that is, $K(X,Y) = K(Y,X)$.
We define the $(0,3)$-tensor field $C$ by $C(X, Y, Z) = (\nabla_X g) (Y, Z)$. This tensor field $C$ is called the \emph{cubic form} of $(g,\nabla)$.
Note that $C(X,Y,Z) = - 2 g(K(X,Y), Z)$ holds.

We introduce the following notions.
Let $\mathcal{A}_{\mathrm{Stat}}(M,g)$ be the set of all statistical connections on the Riemannian manifold $(M, g)$, and $\mathcal{K}(M,g)$ the set of all symmetric $(1,2)$-tensor fields on $M$ satisfying the following equalities,
\[
    g(K(X,Y), Z) = g(K(Y,Z), X) =g(K(Z,X), Y).
\]
In addition, we denote by $S^3(T^\ast M)$ the set of all symmetric $(0,3)$-tensor fields on $M$.
The following fact can be proved directly.

\begin{Prop}\label{Prop_Stat-str}
The following maps are both bijective:
\begin{align*}
    \mathcal{A}_{\mathrm{Stat}}(M,g) \ni \nabla  &\longmapsto  K^{(g, \nabla)} \in \mathcal{K}(M,g),\\
    \mathcal{K}(M,g) \ni K &\longmapsto (\nabla^g + K) g \in S^3(T^\ast M).
\end{align*}
Here, $K^{(g, \nabla)}$ is the difference tensor of $(g, \nabla)$.
The composition of the above maps can be expressed as follows:
\[
    \mathcal{A}_{\mathrm{Stat}}(M,g) \ni \nabla  \longmapsto \nabla g \in S^3(T^\ast M). 
\]
\end{Prop}

By Proposition~\ref{Prop_Stat-str}, a pair consisting of the Riemannian metric $g$ and a cubic form $C \in S^3(T^\ast M)$, or a pair consisting of $g$ and a difference tensor $K \in \mathcal{K}(M,g)$, can also be called a statistical structure on $M$.
Thus, we will also call $(g, C)$ and $(g, K)$ statistical structures on $M$ as well.
We see that the statistical structure $(g, 0)$ on $M$ corresponds to the statistical structure $(g, \nabla^g)$. 
Here $0$ is the origin of $\mathcal{K}(M, g)$ or $S^3(T^\ast M)$ as the linear spaces.

We note that if two statistical manifolds $(M, g, \nabla)$ and $(M', g', \nabla')$
are statistically isomorphic up to scaling by $(f, r)$,
then the following equations hold:
\[
    f^\ast K' = K, \quad f^\ast C' = r \cdot C.
\]
Here, $K$, $C$, $K'$, and $C'$ are the difference tensors and cubic forms on $M$ and $M'$, respectively. 

Throughout this paper, 
for each $K \in \mathcal{K}(M,g)$ and each vector field $X$ on $M$, we define the $(1,1)$-tensor field $K_X$ on $M$ by putting $K_X(Y) := K(X,Y)$ for each vector field $Y$.
For a pair $X,Y$ of vector fields on $M$, 
the Lie bracket $[K_X,K_Y]$ of $K_X$ and $K_Y$ is defined as $[K_X,K_Y](Z) = K_X(K_Y(Z)) - K_Y(K_X(Z))$.
The following proposition, proved by Opozda \cite[Corollary~3.8]{O-2016}, will be applied in Section~\ref{subsec_LISS-Rn}:

\begin{Prop}\label{proposition:Opozda}
Fix a statistical structure $(g,\nabla)$ and a point $p \in M$.
We denote by $K$ the difference tensor of $(g,\nabla)$.
Then the following two conditions on $(g,\nabla,p)$ are equivalent:
\begin{enumerate}[label=(\roman*)]
    \item $[K_X,K_Y]_p = 0$ for any pair $(X,Y)$ of vector fields on $M$.
    \item There exists an orthonormal basis $e_1,\dots,e_n$ of $(T_pM,g_p)$ and $c_1, \dots, c_n \in \RR$ such that 
    \[
    K_p(e_i,e_i) = c_i e_i, \quad K_p(e_i,e_j) = 0
    \]
    for distinct $i,j = 1,\dots,n$.
\end{enumerate}
\end{Prop}

Let $(M_i, g_i, \nabla_i)$ be statistical manifolds for $i = 1, 2$.  
The direct product manifold $M_1 \times M_2$ carries a natural statistical structure given by the direct product metric $g_1 \oplus g_2$ and the direct product connection $\nabla_1 \oplus \nabla_2$.  
The cubic form $C$ of this statistical structure is given by  
\[
C(X,Y,Z) = C_1(X_1, Y_1, Z_1) + C_2(X_2, Y_2, Z_2),
\]
for all vector fields $X, Y, Z$ on $M_1 \times M_2$,  
where $C_i$ denotes the cubic form of $(M_i, g_i, \nabla_i)$, and $X_i$ denotes the $M_i$-component of the vector field $X$ $(i = 1, 2)$.

\subsection{Some classes of statistical structures}\label{subsec:some-class-of-SS}
In this subsection, we introduce four classes of statistical structures: conjugate symmetric, constant curvature, dually flat, and constant Hessian curvature (abbreviated as CHC).
Note that these classes are nested in the order listed above; that is, every CHC structure is dually flat, every dually flat structure has constant curvature, and every constant curvature statistical structure is conjugate symmetric.

In \cite{L-SM}, Lauritzen defined a statistical manifold $(M,g, \nabla)$ to be \emph{conjugate symmetric} if $R^{\nabla} \equiv R^{\overline{\nabla}}$.
We also call a statistical structure $(g, \nabla)$ \emph{conjugate symmetric} if the statistical manifold $(M,g, \nabla)$ is conjugate symmetric.
Note that for each Riemannian metric $g$ on a manifold $M$, 
the trivial statistical structure $(g,\nabla^g)$ is clearly conjugate symmetric. 
Moreover, the property of being conjugate symmetric is closed under taking direct products in the sense of Section~\ref{Sec_SM}.

The conjugate symmetry of a statistical structure can be characterized as follows:

\begin{Prop}[Proposition 2.4 in~\cite{KO-CC}]\label{Prop_K-CS}
Let $(g, \nabla)$ be a statistical structure on $M$,
and let $K$ and $C$ be the difference tensor and the cubic form of $(g, \nabla)$, respectively.
Then, the following four conditions are equivalent:
\begin{enumerate}[label=(\roman*)]
\item The statistical structure $(g, \nabla)$ is conjugate symmetric.
\item The $(0,4)$-tensor field $\nabla C$ is totally symmetric.
\item The $(0,4)$-tensor field $\nabla^g C$ is totally symmetric.
\item The $(1,3)$-tensor field $\nabla^g K$ is totally symmetric, that is, 
\[
(\nabla^g_{X_1} K)(X_2,X_3) = (\nabla^g_{X_{\sigma(1)}} K)(X_{\sigma(2)}, X_{\sigma(3)})
\] for vector fields $X_1,X_2,X_3$ and $\sigma \in \mathfrak{S}_3$.
\end{enumerate}
\end{Prop}

Note that, by Proposition~\ref{Prop_K-CS}, the following set is a linear subspace of $S^3(T^\ast M)$:
\begin{equation}\label{eq_cubic-linsub}
    \left\{ C \in S^3(T^\ast M) ~ \middle| ~ (g, C)  \text{ is conjugate symmetric} \right\}.
\end{equation}

Throughout this paper, for a conjugate symmetric statistical manifold $(M,g,\nabla)$, 
the $(0,4)$-tensor field $R^{\nabla}_g$ is defined by putting 
\[
R^{\nabla}_g(X,Y,Z,W) := g(R^{\nabla}(X,Y)Z,W) 
\]
for vector fields $X,Y,Z,W$ on $M$.
Then the following holds:

\begin{Prop}[cf.~\cite{Opozda_2015}]\label{Prop_curvature-CS}
For a conjugate symmetric statistical manifold $(M,g,\nabla)$, the equality below holds:
\[
    R^\nabla(X,Y) = R^{\nabla^g} (X, Y) + [K_X, K_Y]
\]
for all vector fields $X$, $Y$ on $M$.
In particular, 
the $(0,4)$-tensor field $R^\nabla_g$ satisfies 
\[
R^\nabla_g(X,Y,Z,W) = -R^\nabla_g(Y,X,Z,W) = -R^\nabla_g(X,Y,W,Z) = R^\nabla_g(Z,W,X,Y) 
\]
and 
\[
R^\nabla_g(X,Y,Z,W) + R^\nabla_g(Y,Z,X,W) + R^\nabla_g(Z,X,Y,W) = 0
\]
for vector fields $X,Y,Z,W$ on $M$.
\end{Prop}

For conjugate symmetric statistical manifolds, ``sectional curvatures of the statistical structure'' can be defined in the natural sense as below: 
Let $(M,g,\nabla)$ be a conjugate symmetric statistical manifold and fix a point $p$. 
For each linearly independent pair $\{ v,w \}$ of vectors in $T_pM$, 
we put 
\[
\mathrm{Sect}^\nabla_g(v,w) := \frac{( R^\nabla_g)_p(v,w,w,v)}{g_p(v,v) g_p(w,w)- g_p(v,w)^2}.
\]
Then $\mathrm{Sect}^\nabla_g(v,w)$ is known to depend only on $\mathrm{Span}\{ v,w \}$.
Thus $\mathrm{Sect}^\nabla_g(v,w)$ is called the sectional curvature of the two-dimensional tangential plane $\mathrm{Span}\{ v,w \} \subset T_pM$ on $(M,g,\nabla)$ at $p$ (see \cite{L-SM} for details).

In this paper, 
a statistical manifold $(M,g,\nabla)$ is said to have \emph{constant curvature} if 
there exists a constant $k \in \mathbb{R}$ such that
\[
R^\nabla(X,Y)Z = k(g(Y,Z)X-g(X,Z)Y) 
\]
(see \cite{Kurose-1990}) or equivalently, 
$(M,g,\nabla)$ is conjugate symmetric and 
$\mathrm{Sect}^\nabla_g(v,w) = k$ for any $p \in M$ and any linearly independent pair $(v,w)$ of vectors in $T_pM$.

\begin{Rem}
Opozda~\cite{Opozda_2015} gave a definition of sectional curvatures for general statistical manifolds.
\end{Rem}

A statistical manifold $(M, g, \nabla)$ is called a \emph{dually flat statistical manifold}, or simply \emph{dually flat}, when $R^\nabla = 0$ (or equivalently, $R^{\overline{\nabla}} = 0$ by Equation~\eqref{eq_curv-dual-2}) holds.
For a dually flat statistical manifold $(M, g, \nabla)$, we call $(g, \nabla)$ a \emph{dually flat structure} on $M$.
Note that the dually flat property is preserved under taking direct products, in the sense of Section~\ref{Sec_SM} (see~\cite{Todjihounde_2006}).
As is well-known, dually flat statistical manifolds and dually flat structures are equivalent to the notion of Hessian manifolds and Hessian structures, respectively, see \cite{Shima} for details.
Dually flat structures play an important role in information geometry (see, for example, \cite{AN}).

Let $(M, g, \nabla)$ be a Hessian manifold (i.e.~a dually flat statistical manifold).
A Hessian structure $(g, \nabla)$ is said to be \emph{of CHC} \emph{(of constant Hessian sectional curvature)} if there exists a constant $c \in \RR$ such that the equation
\[
    (\nabla_X K) (Y, Z) = -\frac{c}{2} \{g(X, Y )Z + g(X, Z)Y \}
\]
holds for any vector fields $X, Y, Z$ on $M$, and then $(g, \nabla)$ is also called of CHC $c$ (cf.~\cite{FK_2013, Shima_1995}).
In \cite{FK_2013}, Furuhata and Kurose gave a classification of connected and simply-connected Hessian manifolds of non-positive CHC.

It should be emphasized that  
by the arguments in Section~\ref{Sec_SM}, the following holds.

\begin{Thm}\label{Thm_DF-CS-preser}
The classes of conjugate symmetric, constant curvature, dually flat, and CHC statistical manifolds are preserved under scaling statistical isomorphisms.  
\end{Thm}

\subsection{Some examples of statistical manifolds}\label{Subsec_Ex-SM}
In this subsection, we introduce some examples of statistical manifolds.

\begin{Ex}[$n$-variable normal distribution family]\label{Ex_multi-Normal}
Let $n \in \ZZ$ with $n \geq 1$, and $\mathrm{Sym}^+(n, \RR)$ be the space of all positive definite symmetric matrices of order $n$.
The \emph{$n$-variate normal distribution family} is a statistical model $(\mathcal{N}^n, \RR^n, \mathbf{p}_N)$ with $ \mathcal{N}^n = \mathrm{Sym}^{+}(n, \RR) \times \RR^n$ and
\[
    \mathbf{p}_N: \mathcal{N}^n \to \mathcal{P}(\RR^n), ~ (\Sigma, \mu) \mapsto N(x \mid \Sigma, \mu)dx := \frac{1}{\sqrt{(2\pi )^n \det(\Sigma)}} \exp \bigg{(} -\frac{1}{2} (x - \mu)^{\mathsf{T}} \Sigma^{-1} (x - \mu) \bigg{)}dx
\]
(see \cite{Ay_2017} for the definition of a statistical model).
Here, $dx$ denotes the Lebesgue measure on $\RR^n$, and $\mathcal{P}(\RR^n)$ denotes the space of all probability measures on $\RR^n$.
The \emph{Fisher metric on $(\mathcal{N}^n, \RR^n, \mathbf{p}_N)$} is the Riemannian metric on $\mathcal{N}^n$ defined by
\begin{equation}\label{eq_Fisher-metric}
    g^F_{\hat{\theta}}(X, Y) = \int_{x \in \RR^n } (X \log N(x \mid \theta)) (Y \log N(x \mid \theta)) N(x \mid \hat{\theta}) dx \quad (\hat{\theta} \in \mathcal{N}^n).
\end{equation}
For $\alpha \in \RR$, the \emph{Amari--Chentsov $\alpha$-tensor field on $(\mathcal{N}^n, \RR^n, \mathbf{p}_N)$} is the symmetric $(0,3)$-tensor field on $\mathcal{N}^n$ defined by
\begin{equation}\label{eq_Amari--Chentsov}
    C^{A(\alpha)}_{\hat{\theta}}(X, Y, Z) = \alpha \int_{x \in \RR^n } (X \log N(x \mid \theta)) (Y \log N(x \mid \theta)) (Z \log N(x \mid \theta)) N(x \mid \hat{\theta}) dx \quad (\hat{\theta} \in \mathcal{N}^n).
\end{equation}
Moreover, a statistical connection $\nabla^{A(\alpha)}$ on $(\mathcal{N}^n, g^F)$ corresponding to $(g^F, C^{A(\alpha)})$ in the sense of Proposition~\ref{Prop_Stat-str} is called the \emph{Amari--Chentsov $\alpha$-connection on $(\mathcal{N}^n, \RR^n, \mathbf{p}_N)$}.
\end{Ex}

The following fact for the statistical manifold \((\mathcal{N}^n, g^F, C^{A(\alpha)})\) is well-known.

\begin{Prop}[\cite{Amari_1985,L-SM}]\label{Fact_N-CS-DF}
For any \(\alpha \in \mathbb{R}\), the statistical manifold \((\mathcal{N}^n, g^F, \nabla^{A(\alpha)})\) is conjugate symmetric.  
Moreover, it is dually flat if and only if \(\alpha = \pm 1\).
\end{Prop}

\begin{Ex}[Takano Gaussian space]\label{Ex_Takano}
Let $n \in \ZZ$ with $n \geq 1$, $\mathcal{N}_T^n := \RR_{>0} \times \RR^n$ and 
\[
    \mathbf{p}_T = \mathbf{p}_N \circ \iota : \mathcal{N}^n_T \to \mathcal{P}(\RR^n), ~ (\Sigma, \mu) \mapsto N(x \mid \mathrm{diag}(\Sigma, \dots, \Sigma), \mu)dx.
\]
Here, 
\[
    \iota :  \mathcal{N}^n_T \hookrightarrow \mathcal{N}^n, ~ (\Sigma, \mu) \mapsto (\mathrm{diag}(\Sigma, \dots, \Sigma), \mu).
\]
Then the triple $(\mathcal{N}_T^n, \RR^n, \mathbf{p}_T)$ is a statistical model in the sense of \cite{Ay_2017}. 
On $(\mathcal{N}_T^n, \RR^n, \mathbf{p}_T)$, the Fisher metric \(g^F\), the Amari--Chentsov $\alpha$-tensor field $C^{A(\alpha)}$, and the Amari--Chentsov \(\alpha\)-connection \(\nabla^{A(\alpha)}\) are defined in a similar way as in Example~\ref{Ex_multi-Normal}, respectively.
In \cite{Fujitani}, the Riemannian manifold $(\mathcal{N}_T^n, g^F)$ is called the \emph{Takano Gaussian space}, and we use this terminology in this paper as well.
\end{Ex}

\begin{Rem}
For $n=1$, the statistical model $(\mathcal{N}_T^1,\mathbb{R},\mathbf{p}_T)$ coincides with the univariate normal distribution family $(\mathcal{N}^1,\mathbb{R},\mathbf{p}_{N})$.
\end{Rem}

For the statistical manifolds $(\mathcal{N}^n_T, g^F, \nabla^{A(\alpha)})$, the following holds.

\begin{Prop}\label{Prop:Takano_constant-curvature}
For any \(\alpha \in \mathbb{R}\), the statistical manifold \((\mathcal{N}^n_T, g^F, C^{A(\alpha)})\) has constant curvature (cf. \cite{Takano-2006}). 
Moreover, it is dually flat if and only if \(\alpha = \pm 1\) (cf. \cite{Amari_1985,L-SM}).
\end{Prop}

For other examples of statistical manifolds obtained from statistical models, see \cite{AN, L-SM} for example.

\begin{Rem}
The statistical structure on $\mathcal{N}^n$ defined above is related to statistical inference on the multivariate normal distribution family $(\mathcal{N}^n, \mathbb{R}^n, \mathbf{p}_N)$ (see \cite{AN} for details).
\end{Rem}

\section{Preliminaries for left-invariant geometric structures on Lie groups}\label{Sec_Lie}
The purpose of this section is to fix our terminologies for left-invariant Riemannian metrics and left-invariant statistical structures on Lie groups.

\subsection{Homogeneous statistical manifolds}\label{subsec:homog-SM}
Let $(g, \nabla)$ be a statistical structure on a manifold $M$ equipped with a smooth action of a Lie group $G$.
We say that $(g,\nabla)$ is \emph{$G$-invariant} if $g$ and $\nabla$ are both $G$-invariant.
In particular, when $M$ is a $G$-homogeneous space, the triple $(M, g, \nabla)$ is called a \emph{$G$-homogeneous statistical manifold}.

We shall introduce some examples of homogeneous statistical manifolds obtained from statistical models. We use the notions introduced in Section~\ref{Subsec_Ex-SM}.

\begin{Ex}
The statistical manifold \((\mathcal{N}^n, g^F, \nabla^{A(\alpha)})\) (see Example~\ref{Ex_multi-Normal}) is an \(\mathrm{Aff}(n, \mathbb{R})\)-homogeneous statistical manifold for each $\alpha \in \RR$.
Note that the isotropy subgroup of the $\mathrm{Aff}(n, \RR)$-action on $\mathcal{N}^n$ at the point $(I_n, 0)$ is $O(n) \ltimes \{0\}$.
Here, $I_n$ denotes the identity matrix of size $n$.
\end{Ex}

\begin{Ex}\label{Ex_Takano-simply-trans}
Let $\mathrm{Aff}^{d+}(n, \mathbb{R}) \coloneqq \RR_{>0} \ltimes \RR^n$.
The statistical manifold \((\mathcal{N}_T^n, g^F, \nabla^{A(\alpha)})\) (see Example~\ref{Ex_Takano}) is an \(\mathrm{Aff}^{d+}(n, \mathbb{R})\)-homogeneous statistical manifold for each $\alpha \in \RR$.
Note that the $\mathrm{Aff}^{d+}(n,\RR)$-action on $\mathcal{N}_T^n$ is simply-transitive.
\end{Ex}

The nontrivial point in the above two examples is the invariance of the Fisher metric and the Amari--Chentsov $\alpha$-connection.
The invariance of the Fisher metric and the Amari--Chentsov $\alpha$-tensor field on a statistical model is described in detail in \cite{Ay_2017} (see the \textit{generalization of Chentsov's theorem}, Corollary 5.3 in \cite{Ay_2017}).

\subsection{Left-invariant Riemannian metrics and left-invariant statistical structures on Lie groups} 
\label{Subsec_LISS}
Let $G$ be a Lie group.
We consider $G$ is a $G$-homogeneous space equipped with the left translations. 
In this subsection, we introduce some basic facts about the left-invariant statistical structures on $G$.

A left-invariant Riemannian metric on $G$ is a Riemannian metric that is invariant under left translations.
Let $\mathfrak{g}$ denote the Lie algebra of $G$, that is, the Lie algebra consisting of all left-invariant vector fields on $G$.
As is well-known that the space $\Modin(G)$ of all left-invariant Riemannian metrics on $G$ can be regarded as the space $\Modin(\mathfrak{g})$ of all inner products on the vector space $\mathfrak{g}$ in a canonical sense (see \cite{Kobayashi-Nomizu_II} for details). 
Throughout this paper, the space $\Modin(G) \cong \Modin(\mathfrak{g})$ is considered as a smooth manifold in the usual sense. In particular, for each $\LR \in \Modin(\mathfrak{g})$, we have a diffeomorphism $\Modin(\mathfrak{g}) \cong GL(\mathfrak{g})/{O(\mathfrak{g},\LR)}$, where $GL(\mathfrak{g})$ denotes the general linear group of the vector space $\mathfrak{g}$ and $O(\mathfrak{g},\LR)$ the orthogonal group of the inner product space $(\mathfrak{g},\LR)$.
Note that the manifold $\Modin(\mathfrak{g})$ can be considered as a symmetric space.

A left-invariant affine connection on a Lie group $G$ is an affine connection that is invariant under left translations. 
A statistical structure $(g,\nabla)$ on $G$ is said to be \emph{left-invariant} if $g$ and $\nabla$ are both left-invariant.
For such $(g,\nabla)$, the triple $(G, g, \nabla)$ is called a \emph{statistical Lie group}, see~\cite{FIK}.

Throughout this paper, the set of all left-invariant statistical structures on $G$ will be denoted by 
\[
\Lstat(G) := \left\{  \text{ left-invariant statistical structures } (g,\nabla) \text{ on } G \right\}.
\]

Throughout this paper, 
we write $S^k(\mathfrak{g}^\ast) = \overbrace{\mathfrak{g}^* \odot \dots \odot \mathfrak{g}^*}^{k}$ for the $k$-th order symmetric tensors of the dual space $\mathfrak{g}^\ast$ of $\mathfrak{g}$.
For $x_1, \dots, x_k \in \mathfrak{g}^\ast$, 
we realize the element $x_1 x_2 \cdots x_k \in S^k(\mathfrak{g}^{\ast})$ in the following sense:
\begin{multline*}
    (x_1 x_2 \cdots x_k) (X_1, X_2,\cdots,X_k) :=\\ \frac{1}{k!} \sum_{\sigma \in \mathfrak{S}_k} (x_{\sigma(1)} (X_1)) \cdot (x_{\sigma(2)} (X_2)) \cdots  (x_{\sigma(k)} (X_k)) \quad (X_1,\dots,X_k \in \mathfrak{g}).
\end{multline*}
In particular, if we put $n$ the dimension of $G$ 
and take a basis $\{x_1, \dots ,x_n\}$ of $\mathfrak{g}^\ast$, the set $\{ x_{i_1} x_{i_2} \cdots x_{i_k} \mid 1 \leq i_1 \leq i_2 \leq \cdots \leq i_k \leq n \}$ forms a basis of $S^k(\mathfrak{g}^\ast)$, and thus the real vector space $S^k(\mathfrak{g}^\ast)$ is $\binom{n + k-1}{k}$-dimensional. 
In the following sections, we shall often identify $S^k(\mathfrak{g}^\ast)$ with the vector space $\mathrm{Homog}_k(x_1, \dots, x_n)$ of homogeneous polynomials of degree $k$.

We note that for $k = 3$, 
the space $S^3(\mathfrak{g}^\ast)$ is nothing but the space of all left-invariant cubic forms on $G$.

For each $g \in \Modin(G)$, 
we also define 
\begin{itemize}
    \item $\mathcal{LA}_{\mathrm{Stat}}(G, g) := \left\{ \nabla \in \mathcal{A}_{\mathrm{stat}}(G,g) ~\middle|~ \text{$\nabla$ is left-invariant} \right\} \subset \mathcal{A}_{\mathrm{stat}}(G,g)$ and 
    \item $\mathcal{LK}(G,g) := \left\{ K \in \mathcal{K}(G,g) ~\middle|~ \text{$K$ is left-invariant} \right\} \subset \mathcal{K}(G,g)$.
\end{itemize}

For each $(g,\nabla) \in \Lstat(G)$, 
we denote by $K^{(g,\nabla)}$ the difference tensor of $(g,\nabla)$.
Furthermore, for $g \in \Modin(G)$ and $K \in \mathcal{LK}(G,g)$, let us write $C^{(g,K)}$ for the cubic form of $(g,K)$ in the sense of $C^{(g,K)}(X, Y, Z) := -2g(K(X,Y),Z)$ as in Section~\ref{Sec_SM}.

We also note that for a fixed $g \in \Modin(G)$, 
the space $\mathcal{LK}(G,g)$ is identified with 
the space $\mathcal{K}(\mathfrak{g},g)$ consists of the symmetric bilinear forms $K : \mathfrak{g} \odot \mathfrak{g} \rightarrow \mathfrak{g}$ satisfying the following equalities,
\[
    \langle K(X,Y), Z \rangle = \langle K(Y,Z), X \rangle = \langle K(Z,X), Y \rangle.
\]
where $\LR$ denotes the inner product on the vector space $\mathfrak{g}$ corresponding to the left-invariant metric $g$.

For the sake of computations in sections later, we summarize some correspondences in the form of the following proposition, which is a reformulation of Proposition~\ref{Prop_Stat-str}.

\begin{Prop}\label{Prop_1to1-Lstat}
Let $g$ be a left-invariant Riemannian metric on $G$, and $\LR$ denotes the inner product on $\mathfrak{g}$ corresponding to $g$.
The following maps are bijective:
\begin{align*}
    \mathcal{LA}_{\mathrm{Stat}}(G,g) \ni \nabla  &\longmapsto  K^{(g ,\nabla)} \in \mathcal{LK}(G,g) \cong \mathcal{K}(\mathfrak{g},g),\\
    \mathcal{LK}(G,g) \cong \mathcal{K}(\mathfrak{g},g) \ni K &\longmapsto  C^K \in S^3(\mathfrak{g}^\ast).
\end{align*}
Here, $K^{(g, \nabla)}$ is the difference tensor of $(g, \nabla)$, and $C^{K} (X, Y, Z) := -2\langle K(X, Y ), Z\rangle$.
The composition of the above maps can be expressed as follows:
\[
    \mathcal{LA}_{\mathrm{Stat}}(G,g) \ni \nabla  \longmapsto \nabla g \in S^3(\mathfrak{g}^\ast).
\]
In particular, the composition defines a bijection 
\[
\Lstat(G) = \bigsqcup_{g \in \Modin(G)} \mathcal{LA}_{\mathrm{Stat}}(G,g) \to \Modin(G) \times S^3(\mathfrak{g}^\ast), \quad (g,\nabla) \mapsto (g,C^{(g, K^{(g,\nabla)} )} = \nabla g), 
\]
and hence $\Lstat(G)$ is a trivial vector bundle on the symmetric space $\Modin(G)$ with the fiber $S^3(\mathfrak{g}^\ast)$.  
\end{Prop}

The proof is straightforward. 

Throughout this paper,  we also use the following terminologies:
\begin{align*}
    \LstatCS(G) &\coloneqq \left\{ (g, \nabla) \in \Lstat(G) ~\middle|~ (g, \nabla) \text{ is conjugate symmetric} \right\},\\
    \LstatDF(G) &\coloneqq \left\{ (g, \nabla) \in \Lstat(G) ~\middle|~ (g, \nabla) \text{ is dually flat } \right\}.
\end{align*}
Let us fix $g \in \Modin(G)$ and put:
\begin{itemize}
    \item $S^3_{\mathrm{CS}}(\mathfrak{g}^\ast, g) := \left\{ C \in S^3(\mathfrak{g}^\ast) ~\middle|~ (g, C) \text{ is conjugate symmetric} \right\}$,
    \item $S^3_{\mathrm{DF}}(\mathfrak{g}^\ast, g) := \left\{ C \in S^3(\mathfrak{g}^\ast) ~\middle|~ (g, C) \text{ is dually flat} \right\}$.
\end{itemize}
Note that $S^3_{\mathrm{CS}}(\mathfrak{g}^\ast, g)$ is always a linear subspace of $S^3(\mathfrak{g}^\ast)$ but $S^3_{\mathrm{DF}}(\mathfrak{g}^\ast, g)$ is not.
We also denote these spaces by $S^3_{\mathrm{CS}}(\mathfrak{g}^\ast, \LR)$ and $S^3_{\mathrm{DF}}(\mathfrak{g}^\ast, \LR)$, using the inner product $\LR$ on $\mathfrak{g}$ corresponding to $g$.

\subsection{Some remarks on computations}\label{subsec:remark_computations}

Let $G$ be an $n$-dimensional Lie group and $g$ a left-invariant Riemannian metric on $G$.
In this subsection, we introduce several computational formulas for left-invariant statistical connections on $(G, g)$ that will be used in later sections.

We denote by $\mathfrak{g}$ the Lie algebra of $G$, that is, the Lie algebra of all left-invariant vector fields on $G$, and $\LR$ the inner product on $\mathfrak{g}$ corresponding to $g$.
Let us fix an orthonormal basis $\{ e_1,\dots,e_n \}$ of $\mathfrak{g}$ with respect to $\LR$. 
The structure constants for $\{ e_1,\dots,e_n \}$ are denoted by $a_{ij}^k := \langle [e_i,e_j],e_k \rangle$.

For the Levi-Civita connection $\nabla^g$ on $(G,g)$, 
the (generalized) Christoffel symbols with respect to the global frame $\{ e_1,\dots,e_n \}$ are denoted by $\Gamma_{ij}^{k} := \langle \nabla^g_{e_i} e_j, e_k \rangle$.
Note that $\nabla^g_{e_i} e_j$ is left-invariant on $G$ and thus $\Gamma_{ij}^{k}$ is constant on $G$.

The following is a formula for computing $\Gamma^k_{ij}$ from the structure constants.

\begin{Prop}\label{proposition:Gamma_from_a}
$\Gamma_{ij}^{k} = \frac{1}{2}(a_{ij}^k + a_{ki}^j + a_{kj}^{i})$.
\end{Prop}

\begin{proof}
It is well-known that for each $X,Y \in \mathfrak{g}$, 
\begin{equation*}
\nabla^g_X Y = \frac{1}{2} [X, Y] + U(X,Y)
\end{equation*}
holds, where $U$ is a $\mathfrak{g}$-valued symmetric bilinear form on $\mathfrak{g}$ defined by
\[
2\langle U(X,Y), Z \rangle = \langle [Z, X], Y \rangle + \langle X, [Z,Y]\rangle \quad (X,Y,Z \in \mathfrak{g})
\]
(cf.~\cite{Kobayashi-Nomizu_II}, Chapter X).
By the definition, 
\[
\langle U(e_i,e_j),e_k \rangle 
    = \frac{1}{2}(a_{ki}^j + a_{kj}^i),
\]
and thus 
\begin{align*}
\Gamma_{ij}^k 
    &= \langle \nabla^g_{e_i} e_j, e_k \rangle \\
    &= \langle \frac{1}{2}[e_i,e_j] + U(e_i,e_j), e_k \rangle \\ 
    &= \frac{1}{2} \langle [e_i,e_j], e_k \rangle + \langle U(e_i,e_j), e_k \rangle\\ 
    &= \frac{1}{2}(a_{ij}^k + a_{ki}^j + a_{kj}^{i}).
\end{align*}
\end{proof}

We write $\{ x_1,\dots,x_n \}$ for the dual basis of the orthonormal basis $\{ e_1,\dots,e_n \}$ of $\mathfrak{g}$.
Recall that $\{ x_{i_1} x_{i_2} x_{i_3} \}_{i_1 \leq i_2 \leq i_3}$ forms a basis of $S^3(\mathfrak{g}^\ast)$ (see~Section~\ref{Subsec_LISS}). 
The following is a formula for computing the tensor $\nabla^g(x_{i_1} x_{i_2} x_{i_3}) \in S^3(\mathfrak{g}^\ast) \otimes \mathfrak{g}^\ast$.

\begin{Prop}\label{proposition:nablax_computation}
Let $1 \leq i_1,i_2,i_3 \leq n$.
Then the tensor $\nabla^g(x_{i_1}x_{i_2}x_{i_3}) \in S^3(\mathfrak{g}^\ast) \otimes \mathfrak{g}^\ast$ can be written as 
\[
\nabla^g (x_{i_1} x_{i_2} x_{i_3}) = - \sum_{t=1}^n \sum_{u = 1}^n (\Gamma_{tu}^{i_{1}} (x_u x_{i_{2}} x_{i_{3}}) + 
\Gamma_{tu}^{i_{2}} (x_u x_{i_{1}} x_{i_{3}}) + \Gamma_{tu}^{i_{3}} (x_u x_{i_{1}} x_{i_{2}}))\otimes x_t.
\]
\end{Prop}

The proof is straightforward.

For each $C \in S^3_{\mathrm{CS}}(\mathfrak{g}^\ast,g)$, 
we simply write $R^C_g$ for the left-invariant $(0,4)$ tensor field $R_{g}^{\nabla^{(g,C)}}$ on $G$,
that is, 
\[
R^C_g(X,Y,Z,W) = g(R^{\nabla^{(g,C)}}(X,Y)Z,W),
\]
where $\nabla^{(g,C)}$ denotes the left-invariant statistical connection corresponding to $(g,C)$ in the sense of Proposition~\ref{Prop_1to1-Lstat}.
By Proposition~\ref{Prop_curvature-CS},
we shall regard the left-invariant $(0,4)$-tensor field
$R^C_g$ on $G$ as an element of $(\wedge^2 \mathfrak{g}^\ast) \odot (\wedge^2 \mathfrak{g}^\ast)$, 
where $\wedge^2 \mathfrak{g}^\ast$ denotes the space of alternating $2$-tensors on $\mathfrak{g}^\ast$, 
and $(\wedge^2 \mathfrak{g}^\ast) \odot (\wedge^2 \mathfrak{g}^\ast)$ the space of symmetric $2$-tensors on $(\wedge^2 \mathfrak{g}^\ast)$.
Let us put $\omega_{ij} := x_i \otimes x_j - x_j \otimes x_i$ for each $i,j$.
Then $\{ \omega_{ij} \}_{i < j}$ forms a basis of 
the space $\wedge^2 \mathfrak{g}^\ast$ of alternating $2$-tensors on $\mathfrak{g}^\ast$.
We shall write $(i,j) \leq (k,l)$ if $i < k$ or ``$i=k$ and $j \leq l$''.
Then the family $\{ \omega_{ij} \odot \omega_{kl} \}_{i < j, k < l, (i,j) \leq (k,l)}$ forms a basis of the linear space $(\wedge^2 \mathfrak{g}^\ast) \odot (\wedge^2 \mathfrak{g}^\ast)$.

We note that $\nabla^{(g,0)} = \nabla^g$ and thus $R_g^0$ is the $(0,4)$-curvature tensor of the Riemannian manifold $(G,g)$.
Such a curvature tensor can be written as 
\[
R^0_g = \sum_{i < j, k < l, (i,j) \leq (k,l)} r_{ijkl} (\omega_{ij} \odot \omega_{kl})
\]
with 
\[
r_{ijkl} = 2 R^0_g(e_i,e_j,e_k,e_l) = 2 \sum_{u} (\Gamma_{jk}^u \Gamma_{iu}^{l}  - \Gamma_{ik}^u \Gamma_{ju}^{l} -  a_{ij}^u \Gamma_{uk}^{l})
\]
for $i < j$, $k < l$, $(i,j) < (k,l)$
and 
\[
r_{ijij} = R^0_g(e_i,e_j,e_i,e_j) = \sum_{u} (\Gamma_{ji}^u \Gamma_{iu}^{j}  - \Gamma_{ii}^u \Gamma_{ju}^{j} -  a_{ij}^u \Gamma_{ui}^{j})
\]
for $i < j$.

Throughout this paper, for each $C \in S^3_{\mathrm{CS}}(\mathfrak{g}^\ast, g)$, we denote the $(0,4)$-tensor $R_g^C - R_g^0 \in (\wedge^2 \mathfrak{g}^\ast) \odot (\wedge^2 \mathfrak{g}^\ast)$ by $[K^{(g,C)}, K^{(g,C)}]$, or simply $[K,K]$ when the context is clear.
Then by Proposition~\ref{Prop_curvature-CS}, 
\[
[K,K](X,Y,Z,W) = \langle [K_X,K_Y]Z,W \rangle
\]
holds for $X,Y,Z,W \in \mathfrak{g}$.

By direct calculation, we obtain the following:

\begin{Thm}\label{theorem:Kcurvature}
Let $C \in S^3(\mathfrak{g}^\ast)$. 
For each $u = 1,\dots,n$, we define the matrix $K_u \in M(n,\RR)$ by putting 
\[
(K_u)_{ij} := \langle K^{(g,C)}_{e_u}(e_i), e_j \rangle = -\frac{1}{2}C(e_u,e_i,e_j). 
\]
Then 
\[
    [K^{(g,C)},K^{(g,C)}] = 2 \sum_{i<j,k<l,(i,j) < (k,l)} [K_l,K_k]_{ij} (\omega_{ij} \odot \omega_{kl}) + \sum_{i<j} [K_j,K_i]_{ij} (\omega_{ij} \odot \omega_{ij}),
\]
where $[K_u,K_v] := K_u K_v - K_v K_u$ denotes the Lie bracket of the matrices $K_u$ and $K_v$.
\end{Thm}

\section{Moduli spaces}\label{Sec_LISS}
Throughout this section, let $G$ be an $n$-dimensional Lie group with finitely many connected components, and $\mathfrak{g}$ the Lie algebra of $G$. 
The automorphism group of $G$ and that of $\mathfrak{g}$ are denoted by $\mathrm{Aut}(G)$ and $\mathrm{Aut}(\mathfrak{g})$, respectively.
Note that $\mathrm{Aut}(G)$ and $\mathrm{Aut}(\mathfrak{g})$ are both Lie groups.
The moduli space of all left-invariant statistical structures on $G$ is defined in this section.

\subsection{The moduli space of left-invariant Riemannian metrics}\label{subsec_moduli-LIM}

In this subsection, we recall the definition of moduli spaces of left-invariant Riemannian metrics on Lie groups (cf.~\cite{KTT}).

Throughout this paper, 
we say that two Riemannian manifolds \((M, g)\) and \((M', g')\) are \emph{isometric up to scaling} if there exist a scalar \(r > 0\) and a diffeomorphism \(f: M \to M'\) such that \(r \cdot f^* g' = g\).

As is well-known, the Lie group $\mathbb{R}_{>0} \times \mathrm{Aut}(G)$ acts smoothly on the space of left-invariant Riemannian metrics $\Modin(G)$ (see Section~\ref{Subsec_LISS} for more details on $\Modin(G)$) by
\begin{equation}\label{eq_action-LIM}
    (r, \varphi). g \coloneqq r \cdot (\varphi^{-1})^\ast g = 
    r \cdot g(\varphi^{-1}_\ast ~ \cdot ~ , \varphi^{-1}_\ast ~ \cdot ~ ) \quad ((r , \varphi) \in \mathbb{R}_{>0} \times \mathrm{Aut}(G), ~ g \in \Modin(G)).
\end{equation}
If left-invariant Riemannian metrics $g_1$ and $g_2$ are related by the group action in \eqref{eq_action-LIM}, then the Riemannian manifolds $(G,g_1)$ and $(G, g_2)$ are isometric up to scaling, see Appendix~\ref{Appendix_Moduli-Act} for details.
We call the quotient space $(\mathbb{R}_{>0} \times \mathrm{Aut}(G)) \backslash \Modin(G)$ the \emph{moduli space of left-invariant Riemannian metrics on $G$}, and denote it by $\mathfrak{PM}(G)$.
In what follows, we regard \(\mathfrak{PM}(G)\) as a topological space equipped with the natural quotient topology.

The Lie group $\RR_{>0} \times \mathrm{Aut}(\mathfrak{g})$ also acts smoothly on $\Modin(\mathfrak{g})$ as below:
\begin{equation}\label{eq_action-LIM2}
    (r, \varphi). \LR \coloneqq r \cdot \langle \varphi^{-1} ~ \cdot ~ , \varphi^{-1} ~ \cdot ~ \rangle \quad ((r , \varphi) \in \RR_{>0} \times \mathrm{Aut}(\mathfrak{g}), ~ \LR \in \Modin(\mathfrak{g})).
\end{equation}
Note that if $G$ is connected and simply-connected, 
then $\mathrm{Aut}(G)$ and $\mathrm{Aut}(\mathfrak{g})$ are isomorphic to each other as Lie groups, and furthermore, the homeomorphic correspondence $\Modin(G) \to \Modin(\mathfrak{g})$ is equivariant regarding the actions of $\mathrm{Aut}(G)$ and $\mathrm{Aut}(\mathfrak{g})$.
Thus the moduli space $\mathfrak{PM}(G)$ is homeomorphic to the quotient space $(\RR_{>0} \times \mathrm{Aut}(\mathfrak{g})) \backslash \Modin(\mathfrak{g})$ in such a situation.

\begin{Rem}
In \cite{KTT}, for a Lie group $G$, the orbit space of $\Modin(G)$ under a certain action of the group $\mathbb{R}^\times \times \mathrm{Aut}(G)$ is referred to as the moduli space of left-invariant Riemannian metrics on $G$. 
It is straightforward to see that the action of $\mathbb{R}^\times \times \mathrm{Aut}(G)$ on $\Modin(\mathfrak{g})$ factors through the group $\mathbb{R}_{>0} \times \mathrm{Aut}(G)$ via the natural surjection  
\[
\mathbb{R}^\times \times \mathrm{Aut}(G) \to \mathbb{R}_{>0} \times \mathrm{Aut}(G), \quad (r, \varphi) \mapsto (r^{-2}, \varphi).
\]  
Therefore, the moduli spaces obtained by these two quotient constructions coincide.
\end{Rem}

The connected and simply-connected Lie group with $\mathfrak{PM}(G)$ being a singleton have been classified by \cite{L-DL}.

\begin{Prop}[\cite{L-DL}; see also \cite{KTT}]\label{prop:metric-moduli-singleton}
Let $G$ be a connected and simply-connected Lie group such that $\mathfrak{PM}(G)$ is a singleton.
Then, $G$ is isomorphic to one of the following Lie groups:
\begin{equation}\label{eq_special-Liegrp-2}
    \RR^n, ~ G_{\RR \mathrm{H}^n}\ (n \geq 2) ~ \text{ or } ~ H^3 \times \RR^{n-3}\ (n \geq 3).
\end{equation}
\end{Prop}

The Lie group $\GRHn$ is the solvable part of the Iwasawa decomposition of \(SO_0(n,1)\).
Here, $SO_0(n,1)$ is the identity component of $SO(n,1)$.
Particularly, $\GRHn$ acts simply-transitively on \(n\)-dimensional real hyperbolic space \(\mathbb{R} \mathrm{H}^n \cong \RR_{>0} \times \RR^{n-1}\). 
The Lie group \(G_{\mathbb{R} \mathrm{H}^n}\) is known as the \emph{Lie group of \(n\)-dimensional real hyperbolic space \(\mathbb{R} \mathrm{H}^n\)}. 
The Lie group $H^3$ is the three-dimensional Heisenberg group.

\subsection{The moduli space of left-invariant statistical structures}\label{subsec_moduli-LISS}
In this subsection, we define the moduli space of left-invariant statistical structures on a Lie group $G$.

The Lie group $\RR_{>0 } \times \mathrm{Aut}(G)$ acts smoothly on $\Lstat(G)$ (see Section~\ref{Subsec_LISS} for the notation of $\Lstat(G)$) by
\begin{equation}\label{eq_main-action}
    (r, \varphi). (g, \nabla) \coloneqq (r \cdot (\varphi^{-1})^\ast g, (\varphi^{-1})^\ast \nabla) \quad ((r, \varphi) \in \RRtimesaut(G), ~ (g,\nabla) \in \Lstat(G)).
\end{equation}
We mentioned the smooth structure on $\Lstat(G)$ in Section~\ref{Subsec_LISS} as the trivial vector bundle $\Lstat(G) \cong \Modin(G) \times S^3(\mathfrak{g}^\ast)$.
In this paper, the equivalence relation on $\Lstat(G)$ defined by the group action \eqref{eq_main-action} is denoted by $\approx$, that is $(g, \nabla) \approx (g', \nabla')$ means that there exist a scalar $r > 0 $ and $\varphi \in \mathrm{Aut}(G)$ which satisfies $(r, \varphi). (g, \nabla) = (g', \nabla')$.

We shall introduce the following notion:
\begin{Def}\label{def:ssi_compati-grp}
Let $(g,\nabla)$ and $(g',\nabla')$ be both left-invariant statistical structures on $G$.
We say that a scaling statistical isomorphism $(f,r)$ from $(G,g,\nabla)$ to $(G,g',\nabla')$ (see Definition~\ref{Def_SIUS} for the notation) is \emph{compatible with the group structure} if there exists $\varphi \in \mathrm{Aut}(G)$ such that $f(hx) = \varphi(h) f(x)$ for all $h,x \in G$.
\end{Def}

Then the next proposition holds:
\begin{Prop}\label{proposition:compatible}
Let $(g,\nabla)$ and $(g',\nabla')$ be both left-invariant statistical structures on $G$.
Then \((g, \nabla) \approx (g', \nabla')\) on $\Lstat(G)$ 
holds if and only if 
there exists a scaling statistical isomorphism compatible with the group structure between $(G,g,\nabla)$ and $(G,g',\nabla')$.
\end{Prop}

The proof of Proposition~\ref{proposition:compatible} is postponed to Appendix~\ref{Appendix_Moduli-Act} (as Theorem~\ref{theorem:Statorbit}).

As a corollary to Proposition~\ref{proposition:compatible}, one sees that the group action \eqref{eq_main-action} preserves properties such as conjugate symmetry, constant curvature, the dually flat property, and CHC for left-invariant statistical structures on $G$ (see Section~\ref{subsec:some-class-of-SS} for the definitions of these classes).
\begin{Def}\label{Def_moduli-stat}
The quotient space $(\RR_{>0} \times \mathrm{Aut}(G)) \backslash \Lstat(G)$ by the group action \eqref{eq_main-action} is called the \emph{moduli space of left-invariant statistical structures on $G$}, and denoted by $\MLStat(G)$.
\end{Def}

It should be remarked that the moduli spaces, equipped with the quotient topology, might not be Hausdorff in general.
In fact, the moduli space $\MLStat(G)$ is not Hausdorff when $G = \RR^n$ (see Remark~\ref{Rem_top_moduli-Rn}). 

We also define the \emph{moduli space $\MLStatCS(G)$ of left-invariant
conjugate symmetric statistical structures on $G$} and the
\emph{moduli space $\MLStatDF(G)$ of left-invariant dually flat
structures on $G$} as the following quotient spaces:
\begin{align*}
\MLStatCS(G) &:= (\RR_{>0} \times \mathrm{Aut}(G)) \backslash \LstatCS(G) \text{ and } \\
\MLStatDF(G) &:= (\RR_{>0} \times \mathrm{Aut}(G)) \backslash \LstatDF(G).
\end{align*}

Note that the relative topology of $\MLStatCS(G)$ induced from $\MLStat(G)$ coincides with the quotient topology defined by the surjective map $\LstatCS(G) \to \MLStatCS(G)$.
The topology of $\MLStatDF(G)$ is similar.

\subsection{The cases where $\mathfrak{PM}(G) \cong \{\ast\}$}\label{subsection:modli_unique_metric}
In this subsection, let us assume that 
the moduli space $\mathfrak{PM}(G)$ of left-invariant Riemannian metrics on $G$ is a singleton.
Then  
$\Modin(G)$ is a homogeneous space of $\RR_{>0} \times \mathrm{Aut}(G)$, 
and $\Lstat(G) \cong \Modin(G) \times S^3(\mathfrak{g}^{\ast})$ is an $(\RR_{>0} \times \mathrm{Aut}(G))$-equivariant trivial vector bundle on the homogeneous space $\Modin(G)$.
In such a situation, $\LstatCS(G)$ [resp.~$\LstatDF(G)$] is an $(\RR_{>0} \times \mathrm{Aut}(G))$-equivariant sub-vector-bundle [resp.~sub-fiber-bundle] of the trivial bundle $\Lstat(G)$ on the homogeneous space $\Modin(G)$ with its fiber $S^3_{\mathrm{CS}}(\mathfrak{g}^\ast,g)$ [resp.~$S^3_{\mathrm{DF}}(\mathfrak{g}^\ast,g)$] at each point $g \in \Modin(G)$ (see~Section~\ref{Subsec_LISS} for the definitions of $S^3_{\mathrm{CS}}(\mathfrak{g}^\ast, g)$ and $S^3_{\mathrm{DF}}(\mathfrak{g}^\ast, g)$).

By the general theory of the topology of equivariant fiber bundles on homogeneous spaces (cf.~Theorem~\ref{theorem:fiber_total_homeo} in Appendix~\ref{Appendix_Top-Moduli}),  we obtain the following:

\begin{Prop}\label{Prop_moduli}
Let $G$ be a Lie group with $\mathfrak{PM}(G) \cong \{\ast\}$ (see Section~\ref{subsec_moduli-LIM} for details).
Fix $g \in \Modin(G)$.
Then the following statements hold:
\begin{enumerate}[label=(\arabic*)]
    \item 
    The map 
    \[
    F_g : S^3(\mathfrak{g}^\ast) \to \MLStat(G), ~ C \mapsto (\RR_{>0} \times \mathrm{Aut}(G)).(g, \nabla^{(g, C)})
    \]
    is continuous, open and surjective. Here, $\nabla^{(g, C)}$ is the left-invariant statistical connection corresponding to $(g,C)$ in the sense of Proposition~\ref{Prop_1to1-Lstat}.
    \item The moduli space $\MLStat(G)$ is homeomorphic to the quotient space
\begin{equation}\label{eq_Moduli-Stat-trans}
    (\RR_{>0} \times \mathrm{Aut}(G))^{g} \backslash S^3(\mathfrak{g}^\ast).
\end{equation}
\end{enumerate}
Here, $(\mathbb{R}_{>0} \times \mathrm{Aut}(G))^{g}$ denotes the isotropy subgroup of $\RR_{>0} \times \mathrm{Aut}(G)$ with respect to the action \eqref{eq_action-LIM} at the point $g \in \Modin(G)$, and the Lie group $(\mathbb{R}_{>0} \times \mathrm{Aut}(G))^{g}$ acts on $S^3(\mathfrak{g}^\ast)$ as below:
\begin{equation}\label{eq_action-isotropy-S3}
    (r, \varphi). C \coloneqq r \cdot C(  \varphi^{-1} ~ \cdot ~ , \varphi^{-1} ~ \cdot ~ , \varphi^{-1} ~ \cdot ~ ) \quad ((r, \varphi) \in (\mathbb{R}_{>0} \times \mathrm{Aut}(G))^{g}, ~ C \in S^3(\mathfrak{g}^\ast)),
\end{equation}
where $\varphi^{-1} \cdot$ denotes the induced action of $\varphi^{-1} \in \mathrm{Aut}(G)$ on $\mathfrak{g}$.
Similarly, the moduli spaces $\MLStatCS(G)$ and $\MLStatDF(G)$ are homeomorphic to the quotient spaces
\[
    (\RR_{>0} \times \mathrm{Aut}(G))^{g} \backslash S^3_{\mathrm{CS}}(\mathfrak{g}^\ast, g)
\]
and 
\[
    (\RR_{>0} \times \mathrm{Aut}(G))^{g} \backslash S^3_{\mathrm{DF}}(\mathfrak{g}^\ast, g),
\]
respectively.
\end{Prop}

We also note that 
for the isotropy subgroup $(\RR_{>0} \times \mathrm{Aut}(G))^{g}$, the following is known:

\begin{Prop}[cf.~Lemma~4.2 in Taketomi (2022), arXiv:2210.01483]\label{Prop_isotropy-LIM}
Assume that $G$ is connected and simply-connected (and thus $\mathrm{Aut}(G)$ can be identified with $\mathrm{Aut}(\mathfrak{g})$). 
Fix $g \in \Modin(G)$.
We denote by $\LR$ the corresponding inner product on $\mathfrak{g}$.
Then the claims below hold:
\begin{enumerate}[label=(\arabic*)]
    \item \label{Prop_isotropy-LIM_abelian} If the Lie algebra $\mathfrak{g}$ is abelian, the isotropy subgroup $(\RR_{>0} \times \mathrm{Aut}(G))^{g}$ is isomorphic to the Lie group $\RR_{>0}\times O(\mathfrak{g}, \LR)$.
    \item \label{Prop_isotropy-LIM_nonabelian} If the Lie algebra $\mathfrak{g}$ is non-abelian, the isotropy subgroup $(\RR_{>0} \times \mathrm{Aut}(G))^{g}$ coincides with $\{1\} \times (\mathrm{Aut}(\mathfrak{g}) \cap O(\mathfrak{g}, \LR)) \subset \RR_{>0 }\times \mathrm{Aut}(\mathfrak{g})$.
\end{enumerate}
\end{Prop}

For each Lie group appearing in \eqref{eq_special-Liegrp-2}, the detail of the isotropy subgroup $(\RR_{>0}\times \mathrm{Aut}(G))^{g}$ is described in Sections \ref{subsec_LISS-Rn}, \ref{subsec:LISS-GRHn} and \ref{subsec:LISS-Rn-Heisenberg}, respectively.

\section{The case of $\RR^n$}\label{Sec_Rn}
In this section, we consider the abelian Lie group $\RR^n$ with $n \geq 1$.
Then $\mathfrak{PM}(\RR^n) \cong \{\ast\}$ (see~Proposition~\ref{prop:metric-moduli-singleton}).

Throughout this section, we denote the Lie algebra of $\RR^n$ by $\mathfrak{g}_{\RR^n}$.
Let us fix the standard inner product $\LR$ on $\gee_{\RR^n} \cong \RR^n$, and we denote by $g^E$ the left-invariant Riemannian metric on $\RR^n$ corresponding to $\LR$,
that is, $g^E$ denotes the usual Euclidean metric on $\RR^n$.

\subsection{Main theorem for $\RR^n$}\label{subsec:LISS-Rn}
The goal of this section is to give a proof of the following theorem:

\begin{Thm}\label{thm:Rn-moduli}
The following hold.
\begin{enumerate}[label=(\arabic*)]
    \item \label{item:Rn-moduli_CS} All left-invariant statistical structures on $\RR^n$ are conjugate symmetric, that is $\Lstat(\RR^n) = \LstatCS(\RR^n)$ holds.

    \item \label{item:Rn-moduli_DF} 
    Let $C \in S^3(\mathfrak{g}_{\RR^n}^\ast)$.
    Then the following two conditions on $(\LR, C)$ are equivalent:
        \begin{enumerate}[label=(\roman*)]
            \item \label{item:Rn-moduli_DF_1} The left-invariant statistical structure $(\LR,C)$ on $\RR^n$ is dually flat.
            \item \label{item:Rn-moduli_DF:onb} There exists an orthonormal basis $\{e_1, \dots , e_n\}$ of $(\mathfrak{g}_{\RR^n}, \LR)$ such that $C$ can be expressed as
                \[
                C = \sum_{i=1}^n \lambda_i x_i^3
                \]
            for some $\lambda_i \in \RR$.
            Here, $\{x_1, \dots ,x_n\}$ denotes the dual basis of $\{e_1, \dots ,e_n\}$.
        \end{enumerate}

    \item \label{item:Rn-moduli_moduli_CS}     
    \[
        \MLStat(\RR^n) = \MLStatCS(\RR^n) \cong (\RR_{>0} \times O(n)) \backslash S^3(\mathfrak{g}^\ast_{\RR^n}).
    \]
    \item \label{item:Rn-moduli_moduli_DF}
    Let $\mathcal{B} = \{ e_1,\dots,e_n \}$ be an orthonormal basis of $(\mathfrak{g}_{\RR^n},\LR)$, 
    and we denote by $\{ x_1,\dots,x_n \}$ the dual basis of $\mathcal{B}$.
    Put
    \[
        V_{\mathcal{B}} := \left\{ \sum_{i=1}^n \lambda_i x_i^3 \in S^3(\mathfrak{g}_{\RR^n}^\ast) \,\middle|\, \lambda_i \in \RR \right\} \subset S^3(\gee_{\RR^n}^\ast).
    \]
    Then
    \begin{align*}
        \MLStatDF(\RR^n)
        &\cong (\RR_{>0} \times O(n)) \backslash O(n). V_{\mathcal{B}} \\
        &\cong \RR_{>0} \backslash \{ \lambda \in \RR^n \mid \lambda_1 \geq \dots \geq \lambda_n \geq 0 \}.
    \end{align*}
\end{enumerate}
\end{Thm}

As a corollary to Theorem~\ref{thm:Rn-moduli}~\ref{item:Rn-moduli_DF}, we have the following proposition:
\begin{Prop}
Let $g$ be a left-invariant Riemannian metric on $\RR^n$ with $n \geq 2$. 
Then a left-invariant dually flat structure $(g, \nabla)$ on $\RR^n$ is of CHC (see Section~\ref{subsec:some-class-of-SS} for the notation) if and only if $\nabla$ is the Levi-Civita connection of $g$.
\end{Prop}

\begin{proof}
The Hessian structure $(g, \nabla^g)$ is known to be of dually flat of CHC $0$ on $\RR^n$ (see~\cite{Shima_1995}).
Conversely, we suppose that $(g, \nabla)$ is a left-invariant dually flat structure on $\RR^n$ of CHC, and let us prove that $\nabla = \nabla^g$.
By Proposition~\ref{Prop_curvature-CS} and $R^\nabla = R^{\nabla^g} = 0$, we have $[K_X,K_Y] = 0$ for any $X,Y \in \mathfrak{g}_{\RR^n}$. 
Here, $K$ is the difference tensor of $(g, \nabla)$.
We also note that $(\RR^n,g)$ has the constant sectional curvature $0$.
Thus by \cite[Corollary 2]{Shima_1995} and \cite[Proposition 3.1]{FK_2013}, 
for each $X,Y,Z \in \mathfrak{g}_{\RR^n}$, 
the equality below holds:
\begin{equation}
    K(K(Z,X),Y) = 0 \text{ for } X,Y,Z \in \mathfrak{g}_{\RR^n} \label{eq:FK3-1:4}.
\end{equation}
Take the orthonormal basis $e_1,\dots, e_n$ of $\mathfrak{g}_{\RR^n}$ as in Theorem~\ref{thm:Rn-moduli} \ref{item:Rn-moduli_DF} and write $C = \sum_i \lambda_i x_i^3$.
Note that $K(e_i,e_i) = \lambda_i e_i$ for each $i$.
Therefore, we have
\[
0 = K(K(e_i,e_i),e_i) = \lambda_i^2 e_i.
\]
This proves that $C = 0$ and hence $\nabla = \nabla^g$.
\end{proof}

Let us denote by $g^{E_1}$ the Euclidean metric on the Euclidean line $\RR$.
We fix a unit vector $e \in \mathfrak{g}_{\RR}$ with respect to the inner-product on $\mathfrak{g}_{\RR}$ induced by $g^{E_1}$, 
and $t \in \mathfrak{g}_{\RR}^\ast$ is defined by $t(e) = 1$.
Note that for each $a \in \RR$, 
$at^3$ defines a left-invariant symmetric $(0,3)$-tensor field on $\RR$,  
and $(\RR,g^{E_1},at^3)$ defines a one-dimensional dually flat statistical manifold.

\begin{Thm}
For each $(g,C) \in \LstatDF(\RR^n)$, 
there exists $(a_1,\dots,a_n) \in \RR^n$ such that $(\RR^n,g,C)$ is statistically isomorphic to $\prod_{i=1}^n (\RR,g^{E_1},a_it^3)$, 
where $\prod_{i=1}^n (\RR,g^{E_1},a_it^3)$ denotes the direct product of the family $\{ (\RR,g^{E_1},a_it^3) \}_{i}$ of one-dimensional dually flat statistical manifolds in the sense of Section~\ref{Sec_SM}. 
\end{Thm}

For each $p,r \in \ZZ_{\geq 0}$ with $p+r =n$, 
and each $c = (c_1,\dots,c_n) \in \RR^n$ with $c_1 \geq c_2 \geq \dots c_p > 0$ and $c_{p+1}=\dots=c_{p+r}=0$, let us define the Hessian manifold (i.e.~a dually flat statistical manifold) $(\Omega_c,g_c,\nabla_c)$ as below.
The open domain $\Omega_c$ of the Euclidean space $\RR^n$ is defined by 
\[
\Omega_c := \left\{ (x_1,\dots,x_n) ~\middle|~ x_i > 0 \text{ for each } i \leq p \right\},  
\]
and $\nabla_c$ denotes the Euclidean connection on $\Omega_c$.
The Riemannian metric $g_c$ on $\Omega_c$ is defined as the pullback of the Euclidean metric $g^{E}$ on $\RR^n$ by the diffeomorphism
\[
\phi_c : \Omega_c \rightarrow \RR^n, ~ (x_1,\dots,x_n) \mapsto (\psi_{c_1}(x_1),\dots,\psi_{c_n}(x_n)),
\]
where we put
\[
\psi_{c_i}(x_i) :=     
\begin{cases}
    -\frac{2}{c_i} \log (\frac{c_i}{2} x_i) \text{ if } i \leq p \\
    x_i \text{ if } p < i.
\end{cases}
\]
 
\begin{Cor}\label{cor:1-dim_Hesse-str}
For each $(g,\nabla) \in \LstatDF(\RR^n)$, there exists $c \in \{ \lambda \in \RR^n \mid \lambda_1 \geq \dots \geq \lambda_n \geq 0 \}$ such that the Hessian manifold $(\RR^n,g,\nabla)$ is statistically isomorphic to the Hessian manifold  
$(\Omega_{c},g_{c},\nabla_{c})$ up to scaling.
\end{Cor}

\begin{proof}
By Proposition~\ref{prop:metric-moduli-singleton}, and Theorem~\ref{thm:Rn-moduli}, there exists an orthogonal basis $\mathcal{B} = \{e_1, \dots, e_n\}$ on $(\gee_{\RR^n}, \LR)$ and $c = (c_1, \dots, c_n) \in \{ \lambda \in \RR^n \mid \lambda_1 \geq \dots \lambda_n \geq 0 \}$ such that by putting $C = \sum_{i} c_i x_i^3$, 
the dually flat statistical manifold $(\RR^n,g,\nabla)$ is statistically isomorphic to $(\RR^n,g^E,C)$ up to scaling.
Then, it can be directly verified that the diffeomorphism $\phi_{c}$ gives a scaling statistical isomorphism from the dually flat statistical manifold $(\Omega_c, g_c, \nabla_c)$ onto $(\RR^n,g^E,C)$.
\end{proof}

\begin{Rem}
It should be noted that, in the context of Hessian geometry, the claim of Corollary~\ref{cor:1-dim_Hesse-str} can essentially be found in \cite{Shima-1980, Shima}, where it appears as the classification of homogeneous Hessian domains associated with commutative normal Hessian algebras.
\end{Rem}

\begin{Rem}\label{Rem_top_moduli-Rn}
The moduli space $\MLStatDF(\RR^n)$ is not Hausdorff. In fact, the orbit of the Levi-Civita connections of left-invariant Riemannian metrics and other any orbit can not be separated in the moduli space $\MLStatDF(\RR^n)$ because of the scaling actions of $\RR_{>0}$.
Note that, if we remove the point represented by Levi-Civita connections in the moduli space $\MLStatDF(\RR^n)$, such subspace is Hausdorff. 
\end{Rem}

\subsection{Left-invariant conjugate symmetric statistical structures
and left-invariant dually flat structures on $\RR^n$}\label{subsec_LISS-Rn}
In this subsection, we give a proof of Theorem~\ref{thm:Rn-moduli}.

\begin{proof}[Proof of \ref{item:Rn-moduli_CS} in Theorem~\ref{thm:Rn-moduli}]
By Proposition~\ref{prop:metric-moduli-singleton}, 
our goal is to show that $S^3_{\mathrm{CS}}(\mathfrak{g}_{\RR^n}^\ast,\LR) = S^3(\mathfrak{g}_{\RR^n}^\ast)$.
Fix any $C \in S^3(\mathfrak{g}_{\RR^n}^\ast)$.
Since $\nabla^{g^E}_Y X =0$ for any $X,Y \in \mathfrak{g}_{\RR^n}$,
we have $\nabla^{g^E} C = 0$.
This proves that $C \in S^3_{\mathrm{CS}}(\mathfrak{g}_{\RR^n}^\ast,\LR)$
by Proposition~\ref{Prop_K-CS}.
\end{proof}

\begin{proof}[Proof of \ref{item:Rn-moduli_DF} in Theorem~\ref{thm:Rn-moduli}]
Let $C \in S^3(\mathfrak{g}_{\RR^n}^\ast)$.
Throughout this proof, the difference tensor of the left-invariant statistical structure $(\LR, C)$ is denoted by $K$, and the affine connection by $\nabla$ in the sense of Proposition~\ref{Prop_1to1-Lstat}. 
Moreover, we denote the $(1,3)$-curvature tensor of $\nabla$ by $R$, and the $(1,3)$-curvature tensor of the Levi-Civita connection of $g^E$ by $R_{g^E} \equiv 0$.
\begin{description}
    \item[Proof of \ref{item:Rn-moduli_DF:onb} $\Rightarrow$ \ref{item:Rn-moduli_DF_1}] Assume \ref{item:Rn-moduli_DF:onb}. Then by Proposition~\ref{Prop_1to1-Lstat}, the difference tensor $K$ satisfies 
    \begin{equation}\label{eq_K^C-i,j}
        K(e_i, e_j) = -\frac{1}{2} \delta_{ij} \lambda_i e_i \quad (i,j = 1,\dots ,n),
    \end{equation}
    where $\delta_{ij}$ denotes the Kronecker delta.
    We aim to show that $R$ coincides with $0$.
    By combining $R_{g^E} \equiv 0$, Theorem~\ref{thm:Rn-moduli} \ref{item:Rn-moduli_CS} with the equality Proposition~\ref{Prop_curvature-CS}, we see the following.
    \[
        R(e_i, e_j) = [K_{e_i}, K_{e_j}] \quad (i,j = 1, \dots ,n).
    \]
    By Equation \eqref{eq_K^C-i,j} and Proposition~\ref{proposition:Opozda}, we see that
    \[
        R(e_i, e_j) = 0 \quad (i,j = 1, \dots ,n).
    \]
    Thus, $R \equiv 0$.
        
    \item[Proof of \ref{item:Rn-moduli_DF_1} $\Rightarrow$ \ref{item:Rn-moduli_DF:onb}] Assume \ref{item:Rn-moduli_DF_1}.
    Recall that any dually flat structure is conjugate symmetric. Thus, by Proposition~\ref{Prop_curvature-CS} and $R_{g^E} \equiv 0$, we have the equality 
    \[
        R(X, Y) = [K_X, K_Y] 
    \]
    for all vector fields $X$ and $Y$.
    By assumption (i),
    \[
        [K_u, K_v] \equiv 0
    \]
    holds for all $u$, $v \in \mathfrak{g}_{\RR^n}$.
    By Proposition~\ref{proposition:Opozda}, there exists an orthonormal basis $e_1,\dots,e_n$ of $(\mathfrak{g}_{\RR^n}, \LR)$ and $c_1, \dots, c_n \in \RR$ such that 
    \[
    K(e_i,e_i) = c_i e_i, \quad K(e_i,e_j) = 0
    \]
    for distinct $i,j = 1,\dots,n$.
    Since $C(u,v,w) = -2 \langle K(u,v), w \rangle$ for any $u$, $v$ and $w \in \mathfrak{g}_{\RR^n}$, the following holds:
        \[
        C(e_i, e_j, e_k) = 
            \begin{cases}
                - 2 c_i & (i = j = k) \\
                0        & ( \text{otherwise} ).
            \end{cases}
        \]
    Therefore, by setting $\lambda_i = -2 c_i$, we see that $C$ is expressed as follows:
    \[
        C = \sum_{i=1}^n \lambda_i x_i^3.
    \]
    This completes the proof.
    \end{description}
\end{proof}

\begin{proof}[Proofs of \ref{item:Rn-moduli_moduli_CS} and \ref{item:Rn-moduli_moduli_DF} in Theorem~\ref{thm:Rn-moduli}]
The claim
\[
    \MLStat(\RR^n) = \MLStatCS(\RR^n) \cong (\RR_{>0} \times O(n)) \backslash S^3(\mathfrak{g}^\ast_{\RR^n})
\]
follows from Proposition~\ref{Prop_moduli}, Proposition~\ref{Prop_isotropy-LIM} \ref{Prop_isotropy-LIM_abelian} and Theorem~\ref{thm:Rn-moduli} \ref{item:Rn-moduli_CS}.
The claim
\[
    \MLStatDF(\RR^n)
    \cong (\RR_{>0} \times O(n)) \backslash O(n).V_{\mathcal{B}}
    \cong \RR_{>0} \backslash \{ \lambda \in \RR^n \mid \lambda_1 \geq \dots \geq \lambda_n \geq 0 \}
\]
is obtained by applying Proposition~\ref{Prop_moduli}, Proposition~\ref{Prop_isotropy-LIM} \ref{Prop_isotropy-LIM_abelian}, Theorem~\ref{thm:Rn-moduli} \ref{item:Rn-moduli_DF},
and Theorem~\ref{theorem:RnDFW} in Appendix~\ref{Appendix_Moduli-Rn-DF}.
\end{proof}

\section{The case of $\GRHn$}\label{Sec_GRHn}
In this section, we consider the connected and simply-connected Lie group $\GRHn$ with $n \geq 2$ (see Section~\ref{subsec_moduli-LIM} for the definition of $\GRHn$).
Then $\mathfrak{PM}(G_{\RRH^n}) \cong \{\ast\}$ (see Proposition~\ref{prop:metric-moduli-singleton}).

Throughout this section, we denote the Lie algebra of $\GRHn$ by $\geRHn$. 
Let us fix a \emph{canonical basis} $\{e_1, \dots ,e_n\}$ of $\geRHn$ (see~\cite{KOTT}), that is, $\{e_1, \dots ,e_n\}$ is a basis of $\geRHn$ 
with 
\begin{equation}\label{eq:brackt-relation_grhn}
    [e_1, e_i] = e_i \quad (i = 2, \dots ,n), 
\end{equation}
and the other bracket products are trivial. 
The dual basis of $\{ e_1,\dots,e_n \}$ is denoted by $\{ x_1,\dots,x_n \}$.
Let us consider the inner product $\LR \in \Modin(\geRHn)$ such that the canonical basis is orthonormal.
We denote by $g$ the left-invariant Riemannian metric on $\GRHn$ corresponding to $\LR$.
It is well-known (see~\cite{KOTT}) that the $(0,4)$-curvature tensor $R^0_g$ of $g$ can be written as 
\[
R^0_g = \sum_{1\leq i < j} \omega_{ij} \odot \omega_{ij}
\]
where we put $\omega_{ij} := x_i \otimes x_j - x_j \otimes x_i$. The Riemannian manifold $(G_{\mathbb{R}H^n}, g)$ has constant sectional curvature.

\subsection{Main theorem for $\GRHn$}
The goal of this section is to give a proof of the following theorem:
\begin{Thm}\label{thm:GRHn-cs-df-moduli}
Let us define 
\[
     C^\alpha := \alpha \bigg{(}4 x_1^3 + 6 \sum_{i=2}^{n} x_1 x_i^2 \bigg{)} \in S^3(\mathfrak{g}_{\RRH^n}^\ast) \cong \mathrm{Homog}_3(x_1,\dots,x_n)
\]
for each $\alpha \in \RR$.
Then
\begin{align*}
 \nabla^g C^\alpha &= -6 \alpha (\sum_{i \geq 2} x_i^2)^2 \in S^4(\geRHn^\ast), \\
 R^{C^\alpha}_g &= (1-\alpha^2) R^{0}_g = (1-\alpha^2) \sum_{1\leq i < j} \omega_{ij} \odot \omega_{ij}, \\
 S^3_{\mathrm{CS}}(\mathfrak{g}_{\RRH^n}^\ast,g) &= \{ C^\alpha \mid \alpha \in \RR \},~\text{ and }\\
 S^3_{\mathrm{DF}}(\mathfrak{g}_{\RRH^n}^\ast,g) &= \{ C^{1}, C^{-1} \}.
\end{align*}
Furthermore, 
\begin{align*}
    \MLStat(G_{\RRH^n}) &\cong (SO(1) \times O(n-1)) \backslash \mathrm{Homog}_3(x_1,\dots,x_n), \\
    \MLStatCS(G_{\RRH^n}) &\cong \{ C^\alpha \mid \alpha \in \RR \} \cong \RR,~\text{ and }\\ 
    \MLStatDF(G_{\RRH^n}) & \cong \{ C^\alpha \mid \alpha \in \pm 1 \} \cong \{ \pm 1 \}.
\end{align*}
\end{Thm}

\begin{Rem}
The dually flat structure $(g, C^{-1})$ can be found in \cite[Example 2.3]{Furuhata_2009}, as an example of Hessian structures of CHC $4$ on hyperbolic spaces. Notice that its dual structure $(g,C^{1})$ on $G_{\RRH^n}$ is dually flat but not CHC.
\end{Rem}

As supplementary remarks to Theorem~\ref{thm:GRHn-cs-df-moduli}, we will also discuss the following points in Section~\ref{Subsec_Takano} (see Theorem~\ref{Thm_Amari--Chentsov_Takano}):

\begin{Prop}
For any $\alpha \in \RR$, the left-invariant statistical structure $(g,C^\alpha)$ has constant curvature. 
Furthermore, any left-invariant conjugate symmetric statistical structure on $\GRHn$ is a statistical structure of constant sectional curvature in the sense of Section~\ref{subsec:some-class-of-SS}.
\end{Prop}

\subsection{Left-invariant conjugate symmetric statistical structures
and left-invariant dually flat structures on $\GRHn$}\label{subsec:LISS-GRHn}
In this subsection, we give a proof of Theorem~\ref{thm:GRHn-cs-df-moduli}.

We shall define the subgroup $L$ of $GL(\geRHn)$ by 
\[
L := \mathrm{Aut}(\geRHn) \cap O(\geRHn, \LR) \subset GL(\geRHn).
\]
Since $\mathfrak{PM}(G_{\RRH^n}) \cong \{\ast\}$, 
by applying Propositions~\ref{Prop_moduli} and \ref{Prop_isotropy-LIM} \ref{Prop_isotropy-LIM_nonabelian}, we have 
\begin{align*}
\MLStat(G_{\RRH^n}) &\cong L \backslash S^3(\mathfrak{g}_{\RRH^n}^\ast), \\
\MLStatCS(G_{\RRH^n}) &\cong L \backslash S^3_{\mathrm{CS}}(\mathfrak{g}_{\RRH^n}^\ast,\LR), \\
\MLStatDF(G_{\RRH^n}) &\cong L \backslash S^3_{\mathrm{DF}}(\mathfrak{g}_{\RRH^n}^\ast,\LR).
\end{align*}

Therefore, Theorem~\ref{thm:GRHn-cs-df-moduli} follows from the three propositions below.

\begin{Prop}\label{prop:GRHnL}
We shall identify $GL(\mathfrak{g}_{\RRH^n})$ with 
$GL(n,\RR)$ by considering matrices of linear transformations on $\mathfrak{g}_{\RRH^n}$ with respect to the canonical basis.
Then the subgroup $L := \mathrm{Aut}(\geRHn) \cap O(\geRHn, \LR)$ in $GL(\geRHn) \cong GL(n,\RR)$ can be written as 
\[
L = 
\left\{ 
        \begin{pmatrix}
        1 &    0   \\
        0&   h_0
    \end{pmatrix} 
     \ \middle|~ \ h_0 \in O(n-1)
        \right\} \cong SO(1) \times O(n-1).
\]
Furthermore, $C^\alpha \in S^3(\geRHn^\ast)$ is $L$-invariant for each $\alpha \in \RR$.
\end{Prop}

\begin{Prop}\label{prop:gehrnCScomp}
Let $C \in S^3(\geRHn^\ast) \cong \mathrm{Homog}_3(x_1,\dots,x_n)$.
Then 
$\nabla^g C$ is totally symmetric if and only if $C = C^\alpha$ for some $\alpha \in \RR$.
Furthermore, $\nabla^g C^\alpha = -6 \alpha (\sum_{i \geq 2} x_i^2)^2 \in S^4(\geRHn^\ast)$.
\end{Prop}

\begin{Prop}\label{prop:gehrn-K-curvature}
$R^{C^{\alpha}}_g = (1-\alpha^2)R^{0}_g$ for each $\alpha \in \RR$.
\end{Prop}

First, we shall give a proof of Proposition~\ref{prop:GRHnL}:

\begin{proof}[Proof of Proposition~\ref{prop:GRHnL}]
Under the identification $GL(\geRHn) \cong GL(n,\RR)$, it is well-known that  
the automorphism group $\mathrm{Aut}(\mathfrak{g}_{\RRH^n})$  can be written as 
\[
\mathrm{Aut}(\geRHn) = 
\left\{ 
        \begin{pmatrix}
        1 &    0   \\
        \ast&   h_0  
    \end{pmatrix} 
     \ \middle|~ \ h_0 \in GL(n-1, \RR)
        \right\}.
\]  
Since the canonical basis is orthonormal with respect to the inner product $\LR$, 
under the identification $GL(\mathfrak{g}_{\RRH^n}) \cong GL(n,\RR)$ fixed as above, 
we have 
$O(\mathfrak{g}_{\RRH^n},\LR) \cong O(n)$.
Hence  
\[
L = \mathrm{Aut}(\geRHn) \cap O(n) 
= 
\left\{ 
        \begin{pmatrix}
        1 &    0   \\
        0&   h_0
    \end{pmatrix} 
     \ \middle|~ \ h_0 \in O(n-1)
        \right\} \cong SO(1) \times O(n-1).
\]

It is readily seen that the polynomial function $\sum_{i=2}^{n} x_i^2$ is invariant under the action of the group $L = SO(1) \times O(n-1)$. Therefore, the polynomial function  
\[
C^\alpha = \alpha \left(4 x_1^3 + 6 x_1 \sum_{i=2}^{n} x_i^2 \right)
\]
is also $L$-invariant.
\end{proof}

In order to prove Proposition~\ref{prop:gehrnCScomp}, we prepare the following lemma:
\begin{Lem}\label{lemma:gehrn_nablaC_comp}
Let us put $x_{ijkl} := x_i \otimes x_{j} \otimes x_{k} \otimes x_{l} \in (\gee_{\RRH^n}^\ast)^{\otimes 4}$ for each $1 \leq i,j,k,l \leq n$.
Then 
\begin{align*}
    \nabla^g (x_{1}^3) 
    &= - \sum_{t \geq 2} (x_{11tt} + x_{1t1t} + x_{t11t}),  \\
    \nabla^g (x_{1}^2 x_i) 
    &= x_{111i}  - \frac{1}{3}\sum_{t \geq 2} (x_{t1it} + x_{ti1t} + x_{1tit} + x_{1itt} + x_{it1t} + x_{i1tt}) \quad (\text{for } i \geq 2), \\
     \nabla^g (x_{1} x_i x_j) 
    &= \frac{1}{3}(x_{11ji} + x_{1j1i} + x_{j11i} + x_{11ij} + x_{1i1j} + x_{i11j}) \\ & \quad \quad - \frac{1}{6}\sum_{t \geq 2} (x_{tijt} + x_{tjit} + x_{itjt} + x_{ijtt} + x_{jtit} + x_{jitt}) \quad (\text{for } i,j \geq 2), \text{ and } \\
    \nabla^g(x_i x_j x_k) 
    &= 4 (x_1 \odot x_i \odot x_j \odot x_k) - \frac{1}{6}(x_{ijk1}+x_{ikj1}+x_{jik1} + x_{jki1} +x_{kij1} + x_{kji1}) 
    \quad (\text{for } i,j,k \geq 2)
\end{align*}
hold.
\end{Lem}

\begin{proof}[Proof of Lemma~\ref{lemma:gehrn_nablaC_comp}]
By definition, the structure constants of $\geRHn$ with respect to the canonical basis are given by $a_{1t}^t = 1$, $a_{t1}^t = -1$ ($t \geq 2$), with all other $a_{ij}^k$ equal to zero.
Then, by applying Proposition~\ref{proposition:Gamma_from_a}, the generalized Christoffel symbols can be computed as
\begin{align*}
    \Gamma_{t1}^t &= -1 \quad (t \geq 2), \\
    \Gamma_{tt}^1 &= 1 \quad (t \geq 2), 
\end{align*}
with all other $\Gamma_{ij}^k$ equal to zero.
The claims of Lemma~\ref{lemma:gehrn_nablaC_comp} can be obtained by applying Proposition~\ref{proposition:nablax_computation}.
\end{proof}

Let us give a proof of Proposition~\ref{prop:gehrnCScomp}.
\begin{proof}[Proof of Proposition~\ref{prop:gehrnCScomp}]
First, let us assume that $\nabla^g C$ is totally symmetric.
Put
\[
    \nabla^g C = \sum_{u,v,s,t} ((\nabla^g C)_{uvst}) (x_{uvst})
\]
and $C = \sum_{u,v,s} C_{uvs} (x_{uvs})$,
where we put $x_{uvs} := x_u \otimes x_v \otimes x_s \in (\mathfrak{g}_{\RRH^n}^\ast)^{\otimes 3}$ and $x_{uvst} := x_u \otimes x_v \otimes x_s \otimes x_t \in (\mathfrak{g}_{\RRH^n}^\ast)^{\otimes 4} 
$.
Then by Lemma~\ref{lemma:gehrn_nablaC_comp},
\begin{align*}
0 &= (\nabla^g C)_{i11i} - (\nabla^g C)_{i1i1} = -C_{111} + 2 C_{1ii} \quad (\text{for } i \geq 2), \\
0 &= (\nabla^g C)_{111i} - (\nabla^g C)_{11i1} = 3 C_{11i} \quad (\text{for } i \geq 2), \text{ and }\\
0 &= (\nabla^g C)_{i11j} - (\nabla^g C)_{i1j1} = \frac{1}{3}C_{1ij} \quad (\text{for } 2 \leq i < j).
\end{align*}
Furthermore, by applying $C_{11i} = 0$, we also obtain 
\begin{align*}
    0 &= (\nabla^g C)_{ijk1} - (\nabla^g C)_{ij1k} = -C_{ijk} \quad (\text{for } 2 \leq i \leq j \leq k).
\end{align*}
Thus $C = C_{111} (x_{111} + \frac{3}{2}\sum_{i \geq 2} x_1 x_i^2) = C^{\alpha}$ for $\alpha = C_{111}/4$.
Conversely, by Lemma~\ref{lemma:gehrn_nablaC_comp}, 
one can see that 
\begin{align*}
\nabla^g C^\alpha 
    &= -2 \alpha \sum_{i \geq 2} \sum_{t \geq 2} (x_{tiit} + x_{itit} + x_{iitt}) \\
    &= -2 \alpha \bigg{(}\sum_{i \geq 2} 3x_{iiii} + \sum_{2 \leq i < t}  (x_{iitt}+x_{itit}+x_{itti} + x_{tiit} + x_{titi} + x_{ttii}) \bigg{)} \\
    &= -6 \alpha \bigg{(}\sum_{i \geq 2} x_i^4 + \sum_{2 \leq i < t} 2 x_i^2 x_t^2 \bigg{)} \\
    &= -6 \alpha \bigg{(}\sum_{i \geq 2} x_i^2 \bigg{)}^2
\end{align*}
and in particular, $\nabla^g C^\alpha$ is totally symmetric.
\end{proof}

Finally, we shall prove Proposition~\ref{prop:gehrn-K-curvature}:
\begin{proof}[Proof of Proposition~\ref{prop:gehrn-K-curvature}]
Fix $\alpha \in \RR$, and put $C := C^\alpha$.
Recall that $R^0_g = \sum_{1\leq i < j} \omega_{ij} \odot \omega_{ij}$.
Thus by applying Proposition~\ref{Prop_curvature-CS}, our goal is to show that 
\[
[K,K] = [K^{(g,C)},K^{(g,C)}] = (-\alpha^2) \sum_{1\leq i < j} \omega_{ij} \odot \omega_{ij}
\]
(see Section~\ref{subsec:remark_computations} for the definition of $[K,K]$ and $[K^{(g, C)}, K^{(g, C)}]$).
Put
\[
\sum_{i,j,l} C_{ijl} x_{ijl} = C = \alpha \bigg{(}4 x_{111} + 2 \sum_{i \geq 2} (x_{1ii} + x_{i1i} + x_{ii1}) \bigg{)}.
\]
Then for each $l = 1,\dots,n$, 
the matrix $K_{l} = -\frac{1}{2}(C_{lij})_{ij}$ can be written as 
\[
K_1 = -\alpha \cdot \mathrm{diag}(2,1,\dots,1), \text{ and }
K_u = -\alpha (E_{1u} + E_{u1}) \quad \text{ for } u \geq 2,
\]
where 
we write $E_{ij}$ for the $n \times n$ matrix unit whose $(i,j)$-entry is $1$ and all other entries are zero. 
In particular, for $u \geq 2$, $v \geq 1$,    
\begin{align*}
    [K_u,K_v] &= \begin{cases}
        \alpha^2 (E_{u1} - E_{1u}) \quad \text{ if } v = 1, \\
        \alpha^2 (E_{uv} - E_{vu}) \quad \text{ if } v \geq 2,
    \end{cases} \\
    &= \alpha^2 (E_{uv}-E_{vu}).
\end{align*}
Therefore by Theorem~\ref{theorem:Kcurvature}, 
\[
    [K,K] 
    = (-\alpha^2) \sum_{1 \leq i < j} \omega_{ij} \odot \omega_{ij} = (-\alpha^2) R^0_g.
\]
This completes the proof.
\end{proof}

\subsection{Amari--Chentsov $\alpha$-connection on the Takano Gaussian space}\label{Subsec_Takano}
Throughout this subsection, let $m \in \ZZ_{\geq 1}$.
In this subsection, we present a characterization of the Amari--Chentsov $\alpha$-connection $\nabla^{A(\alpha)}$ on $(\mathcal{N}_T^m, \RR^m, \mathbf{p}_T)$ 
(see~Sections \ref{Subsec_Ex-SM} and \ref{subsec:homog-SM} for notation used in this subsection).

\begin{Thm}\label{Thm_Amari--Chentsov_Takano}
Let $\nabla$ be a statistical connection on the Takano Gaussian space $(\mathcal{N}_T^m, g^F)$.
Then, the following three conditions on $(g^F, \nabla)$ are equivalent:
\begin{enumerate}[label=(\roman*)]
    \item $\nabla$ is $\mathrm{Aff}^{d+}(m, \RR)$-invariant and $(g^F, \nabla)$ is conjugate symmetric.
    \item $\nabla$ is $\mathrm{Aff}^{d+}(m, \RR)$-invariant and $(g^F, \nabla)$ has constant curvature.
    \item $\nabla = \nabla^{A(\alpha)}$ holds for some $\alpha \in \RR$.
\end{enumerate}
\end{Thm}

\begin{proof}
One sees that the Lie group $G_{\RR \mathrm{H}^{m+1}}$ is isomorphic to $\mathrm{Aff}^{d+}(m, \RR)$.
Moreover, \(\mathrm{Aff}^{d+}(m, \mathbb{R})\) has a simply-transitive action on \(\mathcal{N}_T^m\) under which the Fisher metric \(g^F\) and the Amari--Chentsov $\alpha$-connection \(\nabla^{A(\alpha)}\) on $\mathcal{N}_T^m$ are invariant, see Example~\ref{Ex_Takano-simply-trans}.
Thus, by Theorem~\ref{thm:GRHn-cs-df-moduli} and Proposition~\ref{Prop:Takano_constant-curvature}, we obtain the claim.
\end{proof}

In \cite{KO}, the Amari--Chentsov $\alpha$-connection on the $m$-variate normal distribution family $(\mathcal{N}^m, \RR^m, \mathbf{p}_N)$ is characterized in the following sense:
\begin{Prop}[\cite{KO}]\label{Prop_MND-KO}
Let
\[
    R := \left\{T \in GL(m, \mathbb{R}) ~\middle|~ T \text{ is an upper-triangular matrix with all positive diagonal entries} \right\},
\]
\(\mathrm{Aff}^s(m, \mathbb{R}) := R \ltimes \mathbb{R}^m\), and $\nabla$ be a statistical connection on \((\mathcal{N}^m, g^F)\).
Then, the following two conditions on $(g^F, \nabla)$ are equivalent:
\begin{enumerate}[label=(\roman*)]
    \item $\nabla$ is $\mathrm{Aff}^{s}(m, \RR)$-invariant and $(g^F, \nabla)$ is conjugate symmetric.
    \item The affine connection \(\nabla\) is the Amari--Chentsov \(\alpha\)-connection for some $\alpha \in \RR$.
\end{enumerate}
\end{Prop}

Theorem~\ref{Thm_Amari--Chentsov_Takano} provides a characterization similar to the above Proposition~\ref{Prop_MND-KO} for the Amari--Chentsov $\alpha$-connection on $(\mathcal{N}_T^m, \RR^m, \mathbf{p}_T)$.

Let us consider the case where $m = 2$.
In \cite{Inoguchi_2024}, the Lie group $\mathrm{Aff}^{d+}(1, \RR) = \RR_{>0} \ltimes \RR$ is called the \emph{Lie group of (univariate) normal distributions}.
Note that $\mathrm{Aff}^{d+}(1, \RR) \cong G_{\RRH^2}$.
By Theorem~\ref{thm:GRHn-cs-df-moduli}, for the moduli space $\MLStat (\mathrm{Aff}^{d+}(1, \RR))$, the following theorem holds. 

\begin{Thm}
The following is a complete set of representatives for $\MLStat(\mathrm{Aff}^{d+}(1, \RR))$.
\[
S := \left\{(a,b,c,d) ~\middle|~ b>0 \right\}\ \cup\ \left\{(a,0,c,d) ~\middle|~ d\ge0 \right\} \subset S^{3}(\gee_{\RRH^2}^\ast).
\]
Here, $(a,b,c,d) := a x_1^3 + b x_1^2x_2 + cx_1x_2^2 + d x_2^3 \in S^3(\gee_{\RRH^2}^\ast)$, and $\{x_1, x_2\}$ is the dual basis of the canonical basis of $\gee_{\RRH^2}$ (see~Equation~\eqref{eq:brackt-relation_grhn}).
We note that $(a,b,c,d) = (4\alpha, 0, 6\alpha, 0)$ corresponds to $(g^F,\nabla^{A(\alpha)})$.
\end{Thm}

\section{The case of $H^3 \times \RR^{n-3}$}\label{Sec_H3Rn-3}
In this section, we consider the product Lie group of the three-dimensional Heisenberg group $H^3$ and the abelian Lie group $\RR^{n-3}$, where $n \geq 3$. 
Then $\mathfrak{PM}(H^3 \times \RR^{n-3}) \cong \{\ast\}$ (see Proposition~\ref{prop:metric-moduli-singleton}).

Throughout this section, we denote the Lie algebra of $H^3$ by $\haa_3$. 
Let us fix a \emph{canonical basis} $\{e_1, e_2 ,e_3\}$ of $\haa_3$ (see~\cite{KOTT}),
that is, $\{e_1, e_2 ,e_3\}$ is a basis of $\haa_3$ with 
\[
    [e_1, e_2] = e_3, 
\]
and the other bracket products are trivial. 
Furthermore, we fix a basis $\{e_4, \dots, e_n\}$ on the Lie algebra $\gee_{\RR^{n-3}}$ of $\RR^{n-3}$, and we also call the basis $\{e_1, e_2, e_3, e_4, \dots, e_n\}$ of $\haa_3 \oplus \gee_{\RR^{n-3}}$ the \emph{canonical basis}.
The dual basis of $\{ e_1, e_2, e_3, \dots,e_n \}$ is denoted by $\{ x_1, x_2, x_3, \dots, x_n\}$.
Let us consider the inner product $\LR \in \Modin(\haa_3 \oplus \gee_{\RR^{n-3}})$ for which the canonical basis is orthonormal.
We denote by $g$ the left-invariant Riemannian metric on $H^3 \times \RR^{n -3}$ corresponding to $\LR$.
It is well-known (see~\cite{Milnor_1976}) that the $(0,4)$-curvature tensor $R^0_g$ of $g$ on $H^3 \times \RR^{n-3}$ can be written as 
\[
R^0_g = \frac{3}{4}\omega_{12} \odot \omega_{12} - \frac{1}{4} \omega_{13} \odot \omega_{13} - \frac{1}{4} \omega_{23} \odot \omega_{23}
\]
where we put $\omega_{ij} := x_i \otimes x_j - x_j \otimes x_i$.

\subsection{Main theorem for $H^3 \times \RR^{n-3}$}
The goal of this section is to give a proof of the following theorem:

\begin{Thm}\label{thm:H3Rn-3-cs-df-moduli}
Let us define 
\[
    w_0 := x_1 x_1 + x_2 x_2 + x_3 x_3 \in S^2(\mathfrak{h}_3^\ast) \cong  \mathrm{Homog}_2(x_1, x_2, x_3).
\]
Then 
\begin{align*}
 S^3_{\mathrm{CS}}((\haa_3 \oplus \gee_{\RR^{n-3}})^\ast, g) &= \{ w_0 p_1 + p_3 \mid p_1 \in S^1(\gee_{\RR^{n-3}}^\ast), ~  p_3 \in S^3(\gee_{\RR^{n-3}}^\ast) \}, \\
 S^3_{\mathrm{DF}}((\haa_3 \oplus \gee_{\RR^{n-3}})^\ast, g) &= \emptyset.
\end{align*}
Moreover, for each $C \in  S^3_{\mathrm{CS}}((\haa_3 \oplus \gee_{\RR^{n-3}})^\ast, g)$, we have $\nabla^g C = 0$ and 
$\mathrm{Sect}^{\nabla}_{g} (e_1,e_3) \geq 1/4 > 0$, where $\nabla$ is the statistical connection corresponding to $(g, C)$. 
Furthermore, 
\begin{align*}
    \MLStat(H^3 \times \RR^{n-3}) &\cong 
    (S(O(2) \times O(1)) \times O(n-3)) \backslash \mathrm{Homog}_3(x_1,\dots,x_n), \\
    \MLStatCS(H^3 \times \RR^{n-3}) &\cong O(n-3) \backslash ( \mathrm{Homog}_1(x_4, \dots, x_n) \oplus \mathrm{Homog}_3(x_4, \dots, x_n)), \\ 
    \MLStatDF(H^3 \times \RR^{n-3}) &= \emptyset.        
\end{align*}
\end{Thm}

\begin{Rem}
For the case of $H^3$, by Inoguchi and the second author \cite{IO-2024}, there is no non-trivial left-invariant conjugate symmetric statistical structure on $H^3$, that is $S^3_{\mathrm{CS}}(\haa_3^\ast, g) = \{0\}$ holds for any left-invariant Riemannian metric $g \in \Modin(H^3)$.
Moreover, as is well-known, any left-invariant Riemannian metric on \(H^3\) is not flat (see COROLLARY 4.6 in \cite{Milnor_1976}).
Thus, 
we can see that there is no left-invariant dually flat structure on \(H^3\). That is, \(\LstatDF(H^3) = \emptyset\).
\end{Rem}

\begin{Rem}
One can also see that there is no left-invariant constant curvature statistical structure on the three-dimensional Heisenberg group \(H^3\).
Note that a constant curvature statistical structure is conjugate symmetric.
According to \cite{Milnor_1976}, any left-invariant Riemannian metric on a non-abelian nilpotent Lie group is not Einstein, and hence not of constant curvature.
\end{Rem}

As supplementary remarks to Theorem~\ref{thm:H3Rn-3-cs-df-moduli}, we will also discuss the following points:

\begin{Prop}\label{proposition:H3RnotCC}
Any left-invariant conjugate symmetric statistical structure on $H^3 \times \RR^{n-3}$ does not have constant curvature.
\end{Prop}

\subsection{Left-invariant conjugate symmetric statistical structures
and left-invariant dually flat structures on $H^3 \times \RR^{n-3}$}\label{subsec:LISS-Rn-Heisenberg}
In this subsection, we give proofs of Theorem~\ref{thm:H3Rn-3-cs-df-moduli} and Proposition~\ref{proposition:H3RnotCC}.

We shall define the subgroup $L'$ of $GL(\haa_3 \oplus \gee_{\RR^{n-3}})$ by
\[
L' := \mathrm{Aut}(\haa_3 \oplus \gee_{\RR^{n-3}}) \cap O(\haa_3 \oplus \gee_{\RR^{n-3}}, \LR) \subset GL(\haa_3 \oplus \gee_{\RR^{n-3}}).
\]
Since $\mathfrak{PM}(H^3 \times \RR^{n-3}) \cong \{\ast\}$, by applying Propositions~\ref{Prop_moduli} and \ref{Prop_isotropy-LIM} \ref{Prop_isotropy-LIM_nonabelian}, we have 
\begin{align*}
\MLStat(H^3 \times \RR^{n-3}) &\cong L' \backslash S^3((\haa_3 \oplus \gee_{\RR^{n-3}})^\ast), \\
\MLStatCS(H^3 \times \RR^{n-3}) &\cong L' \backslash S^3_{\mathrm{CS}}((\haa_3 \oplus \gee_{\RR^{n-3}})^\ast,\LR), \\
\MLStatDF(H^3 \times \RR^{n-3}) &\cong L' \backslash S^3_{\mathrm{DF}}((\haa_3 \oplus \gee_{\RR^{n-3}})^\ast,\LR).
\end{align*}

Therefore, Theorem~\ref{thm:H3Rn-3-cs-df-moduli} and Proposition~\ref{proposition:H3RnotCC} follow from the three propositions below.

\begin{Prop}\label{prop:H3Rn-3L'}
We shall identify $GL(\haa_3 \oplus \gee_{\RR^{n-3}})$ with $GL(n,\RR)$ by considering matrices of linear transformations on $\haa_3 \oplus \gee_{\RR^{n-3}}$ with respect to the canonical basis.
Then the subgroup $L'$ of $GL(\haa_3 \oplus \gee_{\RR^{n-3}}) \cong GL(n,\RR)$ can be written as 
\begin{equation}\label{eq:H3Rn-3L'}
    L' = \left\{ 
        \begin{pmatrix}
             a&      b&      0&  0&  \cdots&  0 \\
             c&      d&      0&  0&  \cdots&  0 \\
             0&      0&  ad-bc&  0&  \cdots&  0  \\
        \vdots& \vdots&      0&   &        &  \\
        \vdots& \vdots& \vdots&   &      \text{\huge{$\alpha$}}  &  \\
             0&      0&      0&   &        &
    \end{pmatrix} ~\middle|~ 
    \begin{pmatrix}
             a&      b \\
             c&      d      
    \end{pmatrix}
    \in O(2), \ \alpha \in O(n-3)
        \right\} \cong S(O(2) \times O(1)) \times O(n-3).
\end{equation}
Furthermore, $w_0 p_1 \in S^3((\haa_3 \oplus \gee_{\RR^{n-3}})^\ast)$ is $S(O(2) \times O(1))$-invariant for any $p_1 \in S^1(\gee_{\RR^{n-3}}^\ast)$.
\end{Prop}

\begin{Prop}\label{prop:H3Rn-3CScomp}
Let $C \in S^3((\haa_3 \oplus \gee_{\RR^{n-3}})^\ast) \cong \mathrm{Homog}_3(x_1,\dots,x_n)$.
Then $\nabla^g C$ is totally symmetric if and only if $C = w_0 p_1 + p_3$ for some $p_1 \in S^1(\gee_{\RR^{n-3}}^\ast) = \gee_{\RR^{n-3}}^\ast$ and $p_3 \in S^3(\gee_{\RR^{n-3}}^\ast)$. 
Furthermore, $\nabla^g (w_0 p_1 + p_3) = 0$.
\end{Prop}

\begin{Prop}\label{prop:H3Rn-3-curvature}
For each 
\[
C = (x_1^2 + x_2^2 + x_3^2) (c_4 x_4 + \dots + c_n x_n) + \sum_{4 \leq i \leq j \leq k \leq n} C_{ijk} x_i x_j x_k \in S^3((\haa_3 \oplus \gee_{\RR^{n-3}})^\ast)
\]
where $c_4, \dots , c_n, C_{ijk} \in \RR$, 
\begin{align*}
    R^C_g (e_1, e_2, e_2, e_1) &= -\frac{3}{4} + \frac{1}{36} \sum_{k \geq 4} c_k^2, \text{ and } \\
    R^C_g (e_1, e_3, e_3, e_1) &= \frac{1}{4} + \frac{1}{36} \sum_{k \geq 4} c_k^2. 
\end{align*}
\end{Prop}

First, we shall give a proof of Proposition~\ref{prop:H3Rn-3L'}:

\begin{proof}[Proof of Proposition~\ref{prop:H3Rn-3L'}]
Recall that the automorphism group of $\mathfrak{h}_3 \oplus \mathfrak{g}_{\RR^{n-3}}$ has been known (see e.g.~\cite{KTT, KT_2023, KOTT}).
Equation~\eqref{eq:H3Rn-3L'} follows directly from the expressions.
Note that we use the bracket relation $[e_1, e_2] = e_3$, which may have a different numbering of the basis from the previous papers.
It is readily seen that the polynomial function $w_0 =x_1 x_1 + x_2x_2 + x_3 x_3 \in S^2(\haa_3^\ast)$ is invariant under the action of the group $S(O(2) \times O(1))$.
Therefore, the polynomial function $w_0 p_1$ is also $S(O(2) \times O(1))$-invariant for any $p_1 \in S^1(\gee_{\RR^{n-3}}^\ast)$.
\end{proof}

In order to prove Proposition~\ref{prop:H3Rn-3CScomp}, we prepare the following lemma:

\begin{Lem}\label{lemma:H3Rn-3_nablaC_comp}
The following equalities hold.
\begin{align*}
    \nabla^g(x_1^3) &= -\frac{3}{2} (x_1^2 x_2) \otimes x_3 -\frac{3}{2} (x_1^2 x_3) \otimes x_2, \\
\nabla^g(x_2^3) &= \frac{3}{2} (x_2^2 x_1) \otimes x_3 + \frac{3}{2} (x_2^2 x_3) \otimes x_1, \\
\nabla^g(x_3^3) &= -\frac{3}{2} (x_3^2 x_2) \otimes x_1 +\frac{3}{2} (x_3^2 x_1) \otimes x_2, \\
\nabla^g(x_1^2 x_2) &= -(x_1x_2x_3) \otimes x_2 - (x_1 x_2^2) \otimes x_3 + \frac{1}{2} (x_1^2 x_3) \otimes x_1 + \frac{1}{2} (x_1^3) \otimes x_3, \\
\nabla^g(x_1^2 x_3) &= -(x_1 x_3^2) \otimes x_2 - (x_1 x_2 x_3) \otimes x_3 - \frac{1}{2} (x_1^2 x_2) \otimes x_1 + \frac{1}{2} (x_1^3) \otimes x_2, \\
\nabla^g(x_2^2 x_1) &= (x_1x_2x_3) \otimes x_1 + (x_1^2 x_2) \otimes x_3 - \frac{1}{2} (x_2^2 x_3) \otimes x_2 - \frac{1}{2} (x_2^3) \otimes x_3, \\
\nabla^g(x_2^2 x_3) &= (x_2x_3^2) \otimes x_1 + (x_1x_2x_3) \otimes x_3 -  \frac{1}{2} (x_2^3) \otimes x_1+\frac{1}{2} (x_2^2 x_1) \otimes x_2, \\
\nabla^g(x_3^2 x_1) &= -(x_1x_2x_3) \otimes x_1  + (x_1^2 x_3) \otimes x_2 - \frac{1}{2} (x_3^3) \otimes x_2 - \frac{1}{2} (x_2 x_3^2) \otimes x_3, \\
\nabla^g(x_3^2 x_2) &= -(x_2^2 x_3) \otimes x_1 + (x_1x_2x_3) \otimes x_2 +\frac{1}{2} (x_3^3) \otimes x_1 +\frac{1}{2} (x_1 x_3^2) \otimes x_3, \text{ and } \\
\nabla^g(x_1 x_2 x_3) &= -\frac{1}{2}(x_2^2 x_3) \otimes x_3 - \frac{1}{2}(x_2 x_3^2) \otimes x_2 + \frac{1}{2}(x_1^2 x_3) \otimes x_3 + \frac{1}{2}(x_1 x_3^2) \otimes x_1 \\ & \ \ \ +\frac{1}{2} (x_1^2 x_2) \otimes x_2 - \frac{1}{2}(x_1 x_2^2) \otimes x_1.
\end{align*}
Furthermore, for $4 \leq i, j, k \leq n$, the following equalities also hold. 
\begin{align*}
    \nabla^g(x_1^2 x_i) &= -(x_1 x_2 x_i) \otimes x_3 - (x_1 x_3 x_i) \otimes x_2, \\
    \nabla^g(x_2^2 x_i) &= (x_1 x_2 x_i) \otimes x_3 + (x_2 x_3 x_i) \otimes x_1, \\
    \nabla^g(x_3^2 x_i) &= (x_1 x_3 x_i) \otimes x_2 - (x_2 x_3 x_i) \otimes x_1, \\
    \nabla^g(x_1 x_2 x_i) &= -\frac{1}{2}(x_2^2 x_i) \otimes  x_3 - \frac{1}{2} (x_2 x_3 x_i) \otimes x_2 + \frac{1}{2} (x_1^2 x_i) \otimes x_3 + \frac{1}{2} (x_1 x_3 x_i) \otimes x_1, \\
    \nabla^g(x_1 x_3 x_i) &= -\frac{1}{2}(x_2 x_3 x_i) \otimes x_3 - \frac{1}{2} (x_3^2 x_i) \otimes x_2 + \frac{1}{2} (x_1^2 x_i) \otimes x_2 - \frac{1}{2} (x_1 x_2 x_i) \otimes x_1,\\
    \nabla^g(x_2 x_3 x_i) &= \frac{1}{2}(x_1 x_3 x_i) \otimes x_3 + \frac{1}{2}(x_3^2 x_i) \otimes x_1 + \frac{1}{2} (x_1 x_2 x_i) \otimes x_2 - \frac{1}{2} (x_2^2 x_i) \otimes x_1,\\
    \nabla^g(x_1 x_i x_j) &= -\frac{1}{2}(x_2 x_i x_j) \otimes x_3 -\frac{1}{2} (x_3 x_i x_j) \otimes x_2, \\
    \nabla^g(x_2 x_i x_j) &= \frac{1}{2}(x_1 x_i x_j) \otimes x_3 + \frac{1}{2}(x_3 x_i x_j) \otimes x_1, \\
    \nabla^g(x_3 x_i x_j) &= -\frac{1}{2} (x_2 x_i x_j) \otimes x_1 + \frac{1}{2} (x_1 x_i x_j) \otimes x_2, \text{ and } \\
    \nabla^g(x_i x_j x_k) &= 0.
\end{align*}
\end{Lem}

\begin{proof}[Proof of Lemma~\ref{lemma:H3Rn-3_nablaC_comp}]
By definition, the structure constants of $\mathfrak{h}_3 \oplus \mathfrak{g}_{\mathbb{R}^{n-3}}$ with respect to the canonical basis are given by $a_{12}^3 = 1$, $a_{21}^3 = -1$, with all other $a_{ij}^k$ equal to zero.
Then, by applying Proposition~\ref{proposition:Gamma_from_a}, the generalized Christoffel symbols can be computed as
\begin{align*}
\Gamma_{12}^3 &= \Gamma_{23}^1 = \Gamma_{32}^1 = \tfrac{1}{2}, \\ 
\Gamma_{13}^2 &= \Gamma_{21}^3 = \Gamma_{31}^2 = -\tfrac{1}{2},
\end{align*}
with all other $\Gamma_{ij}^k$ equal to zero.
The claims of Lemma~\ref{lemma:H3Rn-3_nablaC_comp} can be obtained by applying Proposition~\ref{proposition:nablax_computation}.
\end{proof}

Let us give a proof of Proposition~\ref{prop:H3Rn-3CScomp}.

\begin{proof}[Proof of Proposition~\ref{prop:H3Rn-3CScomp}]
First, let us assume that $\nabla^g C$ is totally symmetric.
Put 
\[
    \nabla^g C = \sum_{u,v,s,t} ((\nabla^g C)_{uvst}) (x_u \otimes x_v \otimes x_s \otimes x_t)
\]
and $C = \sum_{u,v,s} C_{uvs} (x_u \otimes x_v \otimes x_s)$.
Then by Lemma~\ref{lemma:H3Rn-3_nablaC_comp}, the following claims hold for each $i, j \geq 4$:
\begin{enumerate}[label=(\arabic*)]
    \item $C_{1ij} = 0$ since $(\nabla^g C)_{2ij3} = (\nabla^g C)_{2i3j}$, $(\nabla^g C)_{2ij3} = \alpha C_{1ij}$ (for some $\alpha \neq 0$) and $(\nabla^g C)_{2i3j} = 0$.
    \item $C_{2ij} = 0$ since $(\nabla^g C)_{1ij3} = (\nabla^g C)_{1i3j}$, $(\nabla^g C)_{1ij3} = \alpha C_{2ij}$ (for some $\alpha \neq 0$) and $(\nabla^g C)_{1i3j} = 0$.
    \item $C_{3ij} = 0$ since $(\nabla^g C)_{1ij2} = (\nabla^g C)_{1i2j}$, $(\nabla^g C)_{1ij2} = \alpha C_{3ij}$ (for some $\alpha \neq 0$) and $(\nabla^g C)_{1i2j} = 0$.
    \item $C_{12i} = 0$ since $(\nabla^g C)_{22i3} = (\nabla^g C)_{223i}$, $(\nabla^g C)_{22i3} = \alpha C_{12i}$ (for some $\alpha \neq 0$) and $(\nabla^g C)_{223i} = 0$.
    \item $C_{13i} = 0$ since $(\nabla^g C)_{23i3} = (\nabla^g C)_{233i}$, $(\nabla^g C)_{23i3} =  \alpha C_{13i}$ (for some $\alpha \neq 0$) and $(\nabla^g C)_{233i} = 0$.
    \item $C_{23i} = 0$ since $(\nabla^g C)_{13i3} = (\nabla^g C)_{133i}$, $(\nabla^g C)_{13i3} =  \alpha C_{23i}$ (for some $\alpha \neq 0$) and $(\nabla^g C)_{133i} = 0$.
    \item $C_{11i} = C_{22i}$ since $(\nabla^g C)_{12i3} = (\nabla^g C)_{123i}$, $(\nabla^g C)_{12i3} = \alpha (C_{11i} - C_{22i})$ (for some $\alpha \neq 0$) and $(\nabla^g C)_{123i} = 0$.
    \item $C_{11i} = C_{33i}$ since $(\nabla^g C)_{13i2} = (\nabla^g C)_{132i}$, $(\nabla^g C)_{13i2} = \alpha (C_{11i} - C_{33i})$ (for some $\alpha \neq 0$) and $(\nabla^g C)_{132i} = 0$.
    \item $C_{22i} = C_{33i}$ since $(\nabla^g C)_{23i1} = (\nabla^g C)_{231i}$, $(\nabla^g C)_{23i1} = \alpha (C_{22i} - C_{33i})$ (for some $\alpha \neq 0$) and $(\nabla^g C)_{231i} = 0$.
    \item $C_{123} = 0$ since $(\nabla^g C)_{1122} = (\nabla^g C)_{1221}$, $(\nabla^g C)_{1122} = \alpha C_{123}$ (for some $\alpha > 0$) and $(\nabla^g C)_{1221} = \beta C_{123}$ (for some $\beta < 0$).
    \item $C_{112} = 0$ since $(\nabla^g C)_{1131} = (\nabla^g C)_{1113}$, $(\nabla^g C)_{1131} = \frac{1}{2} C_{112}$ and $(\nabla^g C)_{1113} = \frac{3}{2} C_{112}$.
    \item $C_{113} = 0$ since $(\nabla^g C)_{1121} = (\nabla^g C)_{1112}$, $(\nabla^g C)_{1121} = -\frac{1}{2} C_{113}$ and $(\nabla^g C)_{1112} = \frac{3}{2} C_{113}$.
    \item $C_{122} = 0$ since $(\nabla^g C)_{2232} = (\nabla^g C)_{2223}$, $(\nabla^g C)_{2232} = -\frac{1}{2} C_{122}$ and $(\nabla^g C)_{2223} = -\frac{3}{2} C_{122}$.
    \item $C_{223} = 0$ since $(\nabla^g C)_{2212} = (\nabla^g C)_{2221}$, $(\nabla^g C)_{2212} = \frac{1}{2} C_{223}$ and $(\nabla^g C)_{2221} = -\frac{3}{2} C_{223}$.
    \item $C_{133} = 0$ since $(\nabla^g C)_{3323} = (\nabla^g C)_{3332}$, $(\nabla^g C)_{3323} = -\frac{1}{2}C_{133}$ and $(\nabla^g C)_{3332} = -\frac{3}{2}C_{133}$.
    \item $C_{233} = 0$ since $(\nabla^g C)_{3313} = (\nabla^g C)_{3331}$, $(\nabla^g C)_{3313} = \frac{1}{2}C_{233}$ and $(\nabla^g C)_{3331} = \frac{3}{2}C_{233}$.
    \item $C_{111} = 0$ since $(\nabla^g C)_{1123} = (\nabla^g C)_{1231}$, $(\nabla^g C)_{1123} = -\frac{1}{2}C_{111} + C_{122}$, $(\nabla^g C)_{1231} = \frac{1}{2}C_{122} - \frac{1}{2} C_{133}$, $C_{122} = 0$ and $C_{133} = 0$.    
    \item $C_{222} = 0$ since $(\nabla^g C)_{2213} = (\nabla^g C)_{2132}$, $(\nabla^g C)_{2213} = \frac{1}{2} C_{222} - C_{112}$, $(\nabla^g C)_{2132}= -\frac{1}{2} C_{112} + \frac{1}{2} C_{233}$, $C_{112} = 0$ and $C_{233} = 0$.    
    \item $C_{333} = 0$ since $(\nabla^g C)_{3321} = (\nabla^g C)_{3213}$, $(\nabla^g C)_{3321} = -\frac{1}{2}C_{333} + C_{223}$, $(\nabla^g C)_{3213} = -\frac{1}{2}C_{113} + \frac{1}{2} C_{223}$, $C_{113} = 0$ and $C_{223} = 0$.    
\end{enumerate}
Thus $C$ can be written as $C = w_0 p_1 + p_3$ for some $p_1 \in S^1(\gee_{\RR^{n-3}}^\ast)$ and $p_3 \in S^3(\gee_{\RR^{n-3}}^\ast)$. 
Conversely, let 
\[
    C = (x_1^2 + x_2^2 + x_3^2) (c_4 x_4 + \dots + c_n x_n) + \sum_{4 \leq i \leq j \leq k \leq n} C_{ijk} x_i x_j x_k.
\]
Then, it immediately follows from Lemma~\ref{lemma:H3Rn-3_nablaC_comp} that $\nabla^g C = 0$.
\end{proof}

Finally, we shall prove Proposition~\ref{prop:H3Rn-3-curvature}:

\begin{proof}[Proof of Proposition~\ref{prop:H3Rn-3-curvature}]
Let 
\[
    C = (x_1^2 + x_2^2 + x_3^2) (c_4 x_4 + \dots + c_n x_n) + \sum_{4 \leq i \leq j \leq k \leq n} C_{ijk} x_i x_j x_k.
\]
Recall that 
$R^0_g = \frac{3}{4}\omega_{12} \odot \omega_{12} - \frac{1}{4} \omega_{13} \odot \omega_{13} - \frac{1}{4} \omega_{23} \odot \omega_{23}$ (see~\cite{Milnor_1976}).
Thus by applying Proposition~\ref{Prop_curvature-CS}, our goal is to show that 
\begin{align}
[K,K](e_1, e_2, e_2, e_1) &= [K^{(g, C)}, K^{(g, C)}](e_1, e_2, e_2, e_1) = \frac{1}{36} \sum_{k \geq 4} c_k^2, \label{item:Kcurvature1221}\\
[K,K](e_1, e_3, e_3, e_1) &= [K^{(g, C)}, K^{(g, C)}](e_1, e_3, e_3, e_1) = \frac{1}{36} \sum_{k \geq 4} c_k^2
\label{item:Kcurvature1331}
\end{align}
(see Section~\ref{subsec:remark_computations} for the definitions of $[K,K]$ and $[K^{(g, C)}, K^{(g, C)}]$).
Recall that 
\[
    [K, K] = 2 \sum_{i<j,k<l,(i,j) < (k,l)} [K_l,K_k]_{ij} (\omega_{ij} \odot \omega_{kl}) + \sum_{i<j} [K_j,K_i]_{ij} (\omega_{ij} \odot \omega_{ij})
\]
(cf.~Theorem~\ref{theorem:Kcurvature}).
Thus 
\begin{align*}
    [K, K] (e_1, e_2, e_2, e_1)   =  -[K,K] (e_1, e_2, e_1, e_2) &= - [K_2,K_1]_{12} = [K_1,K_2]_{12},  \\
    [K,K] (e_1, e_3, e_3, e_1) =  -[K,K] (e_1, e_3, e_1, e_3) &= - [K_3,K_1]_{13} = [K_1,K_3]_{13}.
\end{align*}
By directly computing, one can see that 
\[
    K_l = -\frac{1}{6} \sum_{k \geq 4} c_k (E_{lk} + E_{kl})
\]
for each $l = 1,2,3$, 
where we write $E_{ij}$ for the $n \times n$ matrix unit whose $(i,j)$-entry is $1$ and all other entries are zero. 
Therefore, we obtain 
\begin{align*}
    [K_1,K_2] &=  \frac{1}{36} \bigg{(}\sum_{k \geq 4} c_k^2 \bigg{)} (-E_{21}+E_{12}), \\
    [K_1,K_3] &=  \frac{1}{36} \bigg{(}\sum_{k \geq 4} c_k^2 \bigg{)} (-E_{31}+E_{13}), 
\end{align*}
and hence Equations \eqref{item:Kcurvature1221} and \eqref{item:Kcurvature1331} hold.
\end{proof}

\section*{Acknowledgements.}
We would like to acknowledge Hajime Fujita, Hitoshi Furuhata, Shintaro Hashimoto, Hideyuki Ishi, Daisuke Kazukawa, Fumiyasu Komaki, Akira Kubo, Yuichiro Sato, Koichi Tojo, and Masaki Yoshioka for their valuable comments and suggestions.

\appendix
\section{Remarks on topologies of equivariant fiber bundles}\label{Appendix_Top-Moduli}

In this section, we recall a theorem (Theorem~\ref{theorem:fiber_total_homeo})  concerning topologies of equivariant fiber bundles, which is applied in Section~\ref{subsection:modli_unique_metric}.

Let $E,X,F$ be all locally-compact Hausdorff spaces, 
and $\pi : E \rightarrow X$ is a topological $F$-bundle.
We also fix a locally-compact Hausdorff group $G$, and a continuous $G$-action on $E$ and that on $X$ such that $\pi$ is $G$-equivariant. 
Suppose that the $G$-action on $X$ is transitive and 
the map $\varpi_{x} : G \rightarrow X, ~ g \mapsto g.x$ is open for each $x \in X$.
Let us fix $x_0 \in X$.
We write $H$ for the isotropy subgroup of $G$ at $x_0$, 
and $E_0 := \pi^{-1}(x_0)$ the fiber of $\pi$ at $x_0$.
Then $H$ acts on $E_0$ naturally.

The quotient spaces of $E$ and $E_0$ by the $G$-action and by the $H$-action are denoted by $G \backslash E$ and $H \backslash E_0$, respectively.
Let us consider the maps 
\begin{align*}
\psi : E_0 &\rightarrow G \backslash E, ~ c \mapsto [c]_G, \\
\phi : H \backslash E_0 &\rightarrow G \backslash E,~[c]_H \mapsto [c]_G, 
\end{align*}
where $[c]_H$ and $[c]_G$ denote the $H$-orbit and the $G$-orbit through $c \in E_0 \subset E$, respectively.

\begin{Thm}\label{theorem:fiber_total_homeo}
In the setting above, the map $\psi : E_0 \rightarrow G \backslash E$ is a continuous and open surjection.
Furthermore, the map $\phi : H \backslash E_0 \rightarrow G \backslash E$ is a homeomorphism.
\end{Thm}

\begin{proof}
One can easily see that $\psi$ is continuous and surjective, and 
$\phi$ is bijective and continuous.
Our goal is to show that $\psi$ and $\phi$ are both open.
The quotient maps from $E$ and $E_0$ to $G \backslash E$ and $H \backslash E_0$ are denoted by $\theta$ and $\theta_0$.
Since $\theta_0$ is open and $\psi = \phi \circ  \theta_0$, 
we only need to show that $\phi$ is an open map.
Fix any open subset $U$ of $H \backslash E_0$.
We write $V := \theta_0^{-1}(U)$.
Then $V$ is an $H$-stable open subset of $E_0$ and $\theta(V) = \phi(U)$ in $G \backslash E$.
Thus our goal is to show that $G. V = \theta^{-1}(\theta(V))$ is open in $E$.
Fix any $c \in G. V$, and we shall prove that $c$ is an interior point of $G. V$ in $E$.
Put $x := \pi(c) \in X$, $E_x := \pi^{-1}(x)$, $V_x := (G. V) \cap E_x$, and denote by $H_x$ the isotropy subgroup of $G$ at $x$.
Fix $g \in G$ with $g.x_0 = x$.
Then one sees that $H_x = gHg^{-1}$, 
$g.V = V_x$, and hence $V_x$ is an $H_x$-stable open subset of $E_x$ with $c \in V_x$ and $G. V = G. V_x$. 
Let us fix an open neighborhood $W$ of $x$ in $X$ with a trivialization $\eta : \pi^{-1}(W) \rightarrow W \times F$.
We denote by $O := \varpi_{x}^{-1}(W) \subset G$.
Then $O$ is an open neighborhood of the unit $e_G$ of $G$, 
and $O. E_x \subset \pi^{-1}(W)$.

Let us write $\eta_F$ for the composition of $\eta$ and the projection from $W \times F$ onto $F$.
Put $c_F := \eta_F(c) \in F$, $c' := \eta(c) = (x,c_F) \in W \times F$, 
$V_x' := \eta_F(V_x) \subset F$, and $\Omega := \eta(O. V_x) \subset W \times F$.
We only need to show that $c'$ is an interior point of $\Omega$ in $W \times F$.
Since $F$ is locally-compact Hausdorff and $V_x'$ is an open neighborhood of $c_F$ in $F$, 
one can take a compact neighborhood $K$ of $c_F$ in $F$ with $K \subset V_x'$. 
Let us consider the continuous map 
\[
\Phi : O \times F \rightarrow F, ~ (g,v) \mapsto \eta_F(g^{-1}. \eta^{-1}(g.x,v)).
\]
Note that for each $v \in F$ and each $g \in O$, 
the equalities $\Phi(e_G,v) = v$  
and $(g.x,v) = \eta(g. \alpha)$ holds
where we define $\alpha$ to be the unique element in $E_x$ with $\eta_F(\alpha)  = \Phi(g,v) \in F$. 
By Lemma~\ref{lemma:cptopen} below, one can find an open neighborhood $O'$ of $e_G$ in $O$ such that $\Phi(O' \times K) \subset V_x'$.
For such $O'$, one can see that 
\[
c' \in \varpi_x(O') \times K \subset \Omega.
\]
Since $\varpi_x$ is an open map, $\varpi_x(O') \times K$ is a neighborhood of $c'$ in $W \times F$.
This completes the proof.
\end{proof}

\begin{Lem}\label{lemma:cptopen}
Let $O$, $F$ be both topological spaces, $e \in O$, 
and 
$\Phi : O \times F \rightarrow F$ a continuous map with $\Phi(e,v) = v$ for any $v \in F$. 
Fix $v_0 \in F$, a compact neighborhood $K$ of $v_0$ in $F$, 
and an open neighborhood $V$ of $v_0$ in $F$ with $K \subset V$.
Then there exists an open neighborhood $O'$ of $e$ in $O$ such that $\Phi(O' \times K) \subset V$.    
\end{Lem}

\section{Remarks on moduli of geometric structures on homogeneous spaces}\label{Appendix_Moduli-Act}
Let us introduce the following notion for maps between manifolds equipped with Lie group actions:
\begin{Def}
Let $G_i$ be a Lie group acting smoothly on a smooth manifold $M_i$ for each $i = 1,2$.
We say that a pair  
$f : M_1 \rightarrow M_2$ and $\phi : G_1 \rightarrow G_2$ of smooth maps is \emph{$(G_1,G_2)$-equivariant} if $\phi$ is a Lie group homomorphism, 
and 
for any $\eta \in G_1$ and any $x \in M_1$, the equality 
$f(\eta. x) = \phi(\eta). (f(x))$ holds.
\end{Def}

Let us fix a Lie group $G$ and a closed subgroup $H$ of $G$.
The homogeneous space corresponding to $(G,H)$ is denoted by $M$.
In this section, we put $\mathrm{GRiem}(M)$ and $\mathrm{GStat}(M)$ to the set of all $G$-invariant Riemannian metrics on $M$ and that of all $G$-invariant statistical structures on $M$.
Let us define the Lie subgroup 
\[
\mathrm{Aut}_H(G) := \left\{ \varphi \in \mathrm{Aut}(G) ~\middle|~ \varphi(H) = H \right\}
\]
of $\mathrm{Aut}(G)$.
For each $\varphi \in \mathrm{Aut}_H(G)$, 
we define a smooth diffeomorphism $\theta_\varphi : M \rightarrow M$ by putting 
\[
\theta_{\varphi}(\eta H) := \varphi(\eta) H \quad (\text{for } \eta \in G).
\]
Then the Lie group $\mathbb{R}_{>0} \times \mathrm{Aut}_H(G)$ acts smoothly on 
$\mathrm{GRiem}(M)$ and $\mathrm{GStat}(M)$ as below:
\begin{align*}
(r,\varphi). g &:= r \cdot ((\theta_\varphi^{-1})^* g), \\
(r,\varphi). (g,\nabla) &:= ((r,\varphi). g, (\theta_\varphi^{-1})^* \nabla).
\end{align*}

In the setting above, the following two theorems hold:

\begin{Thm}\label{theorem:Riemorbit}
Let $g$ and $g'$ be both $G$-invariant Riemannian metrics on $M$.
Then the following two conditions on $g$ and $g'$ are equivalent:
\begin{enumerate}[label=(\roman*)]
    \item \label{item:Riemorbit:orbit} The metrics $g$ and $g'$ lie in the same $(\mathbb{R}_{>0} \times \mathrm{Aut}_H(G))$-orbit in $\mathrm{GRiem}(M)$. 
    \item \label{item:Riemorbit:isom} There exists a $(G,G)$-equivariant pair $(f : M \rightarrow M, \ \phi : G \rightarrow G)$ of maps such that 
    $\phi \in \mathrm{Aut}(G)$ and 
    $f$ gives an isometry up to scaling between $(M,g)$ and $(M,g')$.
\end{enumerate} 
\end{Thm}

\begin{proof}[Proof of Theorem~\ref{theorem:Riemorbit}]
First, let us assume that there exists $(r,\varphi) \in \mathbb{R}_{>0} \times \mathrm{Aut}_H(G)$ such that $g = (r,\varphi). g'$ in $\mathrm{GRiem}(M)$.
Then one can see that $(\theta_{\varphi}^{-1},\varphi^{-1})$ gives a $(G,G)$-equivariant pair of maps and the pair $(f,r)$ yields an isometry up to scaling between $(M,g)$ and $(M,g')$.
This proves \ref{item:Riemorbit:orbit} $\Rightarrow$ \ref{item:Riemorbit:isom}.

Conversely, suppose \ref{item:Riemorbit:isom}.
Fix $(f,\phi)$ as a $(G,G)$-equivariant pair of maps such that 
$\phi \in \mathrm{Aut}(G)$, $f \in \mathrm{Diff}(M)$, and $r \cdot (f^\ast g') = g$ ($r > 0$).
Let us take $\eta_0 \in G$ with $\eta_0H = f(e_G H)$ in $M$, where $e_G$ denotes the unit of the group $G$.
Then one sees that
\[
f(\eta H) = \phi(\eta) f(e_G H) = \phi(\eta) \eta_0 H  ~ \text{ for any } \eta \in G.
\]
We define a Lie group automorphism $\varphi$ on $G$ by 
\[
\varphi : G \rightarrow G, ~ \eta \mapsto \phi^{-1}(\eta_0 \eta \eta_0^{-1}).
\]
Then $\varphi(H) = H$ since for each $h \in H$, the following equations hold in $M = G/H$:
\[
(\varphi^{-1}(h)) H = (\eta_0^{-1} \phi(h) \eta_0) H = \eta_0^{-1} (f(hH)) = \eta_0^{-1} (f(e_G H)) = \eta_0^{-1} (\eta_0 H) = e_G H.
\]
Furthermore, one sees that $\theta_{\varphi}^{-1} = \rho^M_{\eta_0^{-1}} \circ f : M \rightarrow M$, where
\[
\rho^M : G \rightarrow \mathrm{Diff}(M), ~ \eta \mapsto \rho^M_{\eta}
\]
denotes the $G$-action on $M$.
Since the metric $g'$ is $G$-invariant, we have 
\[
(r,\varphi). g' = r \cdot ((\theta_{\varphi}^{-1})^\ast g') = r \cdot (f^\ast ((\rho^M_{\eta_0^{-1}})^\ast g')) = r \cdot (f^\ast g') = g.
\]
This proves \ref{item:Riemorbit:orbit} $\Leftarrow$ \ref{item:Riemorbit:isom}.
\end{proof}

\begin{Thm}\label{theorem:Statorbit}
Let $(g,\nabla)$ and $(g',\nabla')$ be both $G$-invariant statistical structures on $M$.
Then the following two conditions on $(g,\nabla)$ and $(g',\nabla')$ are equivalent:
\begin{enumerate}[label=(\roman*)]
    \item The statistical structures $(g,\nabla)$ and $(g',\nabla')$ lie in the same $(\mathbb{R}_{>0} \times \mathrm{Aut}_H(G))$-orbit in $\mathrm{GStat}(M)$. 
    \item There exists a $(G,G)$-equivariant pair $(f : M \rightarrow M, \ \phi : G \rightarrow G)$ of maps such that 
    $\phi \in \mathrm{Aut}(G)$ and $f$ gives a scaling statistical isomorphism between $(M,g,\nabla)$ and $(M,g',\nabla')$.
\end{enumerate}    
\end{Thm}

The proof of Theorem~\ref{theorem:Statorbit} can be given similarly to the proof of Theorem~\ref{theorem:Riemorbit}.

Note that Proposition~\ref{proposition:compatible} is a special case of Theorem~\ref{theorem:Statorbit} when $H$ is trivial and 
$M = G$.

\section{Proof of $\MLStatDF(\RR^n) \cong (\RR_{>0}) \backslash \{ \lambda_1 \geq \dots \geq \lambda_n \geq 0 \}$}\label{Appendix_Moduli-Rn-DF}
Let us consider $\RR^n$ the vector space equipped with the standard inner product $\langle \cdot,\cdot \rangle_{\RR^n}$.
The standard basis of $\RR^n$ and its dual basis of $(\RR^n)^\ast$ will be denoted by $\{ e_1,\dots,e_n \}$ and $\{ x_1,\dots,x_n \}$, respectively.
We denote by $S^3((\RR^n)^\ast)$ the vector space of all symmetric $(0,3)$-tensors on $\RR^n$, 
equipped with the $O(n)$-action defined in the usual sense.
The $n$-dimensional subspace $V$ of $S^3((\RR^n)^\ast)$ is defined by 
\[
V := \left\{ \sum_i a_i x_i^3 ~\middle|~ a_i \in \RR \right\}, 
\]
and we write $\Xi := O(n).V \subset S^3((\RR^n)^\ast)$.
Our concern in this section is to determine the quotient space $O(n) \backslash \Xi$.

Let us define the subspace $V_+$ of $V$ by 
\[
V_+ := \left\{ \sum_i \lambda_i x_i^3 ~\middle|~ \lambda_1 \geq \dots \geq \lambda_n \geq 0 \right\}.
\]
The goal of this section is to give a proof of the following theorem:
\begin{Thm}\label{theorem:RnDFW}
The map 
\[
V_+ \rightarrow O(n) \backslash \Xi, ~ \sum_i \lambda_i x_i^3 \mapsto O(n).\left( \sum_i \lambda_i x_i^3 \right)
\]
is a homeomorphism.
\end{Thm}

\begin{Rem}
We believe that $\Xi$ is a smooth submanifold of the vector space $S^3((\RR^n)^\ast)$ and that the natural $O(n)$-action on $\Xi$ is polar with a section $V$.
If this is indeed the case, then Theorem~\ref{theorem:RnDFW} would follow directly from the general theory of polar actions 
(cf. Gorodski (2022), arXiv:2208.03577).
Although we have attempted to verify this claim, a proof has not yet been obtained.
\end{Rem}

Let us consider the standard injective group homomorphism $\mathfrak{S}_n \ltimes \{ \pm 1 \}^n$ into $O(n)$, 
and the image will be denoted by $W \subset O(n)$.
Then $V_+$ is a complete representative of the $W$-action on $V$.

A key part of Theorem~\ref{theorem:RnDFW} is the following:

\begin{Prop}\label{proposition:OnW}
Let us fix $C,C' \in V$ such that $O(n).C = O(n).C'$.
Then $W.C = W.C'$.
\end{Prop}

To prove  
Proposition~\ref{proposition:OnW}, 
we give two lemmas below:

\begin{Lem}\label{lemma:Omega_lambda_od}
For each $\lambda \in \RR^n$, we put 
\[
\support (\lambda) := \left\{ i \in \{ 1,\dots,n \} ~\middle|~ \lambda_i \neq 0 \right\}
\]
and 
\[
\Omega_\lambda :=  \left\{ \sum_{i} a_i e_i \in \RR^n ~\middle|~ a_i \neq 0 \text{ for  } i \in \support (\lambda), \text{and }a_i \lambda_i \neq a_j \lambda_j \text{ for } i,j \in \support (\lambda) \text{ with } i \neq j \right\}.
\]    
Then $\Omega_\lambda$ is open dense in $\RR^n$.
\end{Lem}

\begin{proof}[Proof of Lemma~\ref{lemma:Omega_lambda_od}]
Recall that a finite intersection of open dense subsets of $\RR^n$ is again open and dense.
Since the conditions $a_i \neq 0$ and $a_i \lambda_i \neq a_j \lambda_j$ (for $i \neq j$) each define open dense subsets of $\RR^n$, it follows that $\Omega_\lambda$ is open and dense in $\RR^n$.
\end{proof}

\begin{Lem}\label{lemma:03end}
For each $P \in S^3((\RR^n)^\ast)$ and $u \in \RR^n$, 
let us write 
$\widetilde{K}^P_u : \RR^n \rightarrow \RR^n$ for the unique linear endomorphism satisfying that 
\[
\langle \widetilde{K}^P_u(v),w \rangle_{\RR^n} = \langle P, u \otimes v \otimes w \rangle \quad (\text{for any } v,w \in \RR^n),
\]
where $\langle P, u \otimes v \otimes w \rangle$ denotes the pairing of the $(0,3)$-tensor $P$ and the $(3,0)$-tensor $u \otimes v \otimes w$ on $\RR^n$.
Then for each $u \in \RR^n$ and $k \in O(n)$, the equality  $\widetilde{K}^{P}_{u} = k^{-1} \circ \widetilde{K}^{k.P}_{ku} \circ k$ holds as linear endomorphisms on $\RR^n$.
\end{Lem}

\begin{proof}
Let us put $Q = k.P$ and fix any $u,v,w \in \RR^n$.
Then 
\begin{align*}
\langle \widetilde{K}^P_{u}(v),w \rangle_{\RR^n} 
&= \langle P, u \otimes v \otimes w \rangle \\
&= \langle k^{-1}.Q, u \otimes v \otimes w \rangle \\
&= \langle Q, ku \otimes kv \otimes kw \rangle \\
&= \langle \widetilde{K}^Q_{ku}(kv),kw \rangle_{\RR^n} \\
&= 
\langle k^{-1}\widetilde{K}^Q_{ku}(kv),w \rangle_{\RR^n}
\end{align*}
This proves the claim. 
\end{proof}

\begin{proof}[Proof of Proposition~\ref{proposition:OnW}]
For each $\lambda = (\lambda_1,\dots,\lambda_n) \in \RR^n$, 
we write $C_\lambda = \sum_{i} \lambda_i x_i^3 \in V$.
Fix $\lambda,\mu \in \RR^n$ and $k \in O(n)$ with $k.C_\lambda = C_\mu$.
Then our goal is to find $k' \in W$ such that $k'.C_\lambda = C_\mu$. 
By Lemma~\ref{lemma:Omega_lambda_od}, $\Omega_\lambda \cap k^{-1} \Omega_\mu$ is open dense in $\RR^n$, and in particular non-empty.
Fix $u = \sum_{i} a_i e_i \in \Omega_\lambda \cap k^{-1} \Omega_\mu$.
We write $ku = \sum_{i} b_i e_i \in \Omega_\mu$.
Let us define the linear endomorphisms $\widetilde{K}^{C_\lambda}_u$ and $\widetilde{K}^{C_\mu}_{ku}$ on $\RR^n$ as in Lemma~\ref{lemma:03end}.
Then by the definition, we have 
$\widetilde{K}^{C_\lambda}_u (e_i) = a_i \lambda_i e_i$
and $\widetilde{K}^{C_\mu}_{ku}(e_i) = b_i \mu_i e_i$ for each $i$.
Since $u \in \Omega_\lambda$, 
we see that the family 
\[
\{ \RR e_i \}_{i \in \support (\lambda)} \sqcup \{ \mathrm{Span} \{ e_i \mid \lambda_i = 0 \} \}
\]
gives the eigenspace decomposition of $\RR^n$ by the operator $\widetilde{K}^{C_\lambda}_u$
such that 
the eigenvalue of $\widetilde{K}^{C_\lambda}_u$ on 
$\mathrm{Span} \{ e_i \mid \lambda_i = 0 \}$ is zero.
Furthermore, since $ku \in \Omega_\mu$, 
\[
\{ \RR e_j \}_{j \in \support (\mu)} \sqcup \{ \mathrm{Span} \{ e_j \mid \mu_j = 0 \} \}
\]
gives the eigenspace decomposition of $\RR^n$ by the operator $\widetilde{K}^{C_\mu}_{ku}$
such that 
the eigenvalue of $\widetilde{K}^{C_\mu}_{ku}$ on 
$\mathrm{Span} \{ e_i \mid \mu_i = 0 \}$ is zero.
By Lemma~\ref{lemma:03end},
\[
\widetilde{K}^{C_\lambda}_{u} = k^{-1} \circ \widetilde{K}^{C_\mu}_{ku} \circ k \text{ as } \RR^n \rightarrow \RR^n.
\]
This implies the equations
\[
\mathrm{Span} \{ e_i \mid \lambda_i = 0 \} = k^{-1} \mathrm{Span} \{ e_j \mid \mu_j = 0 \}
\]
and 
\[
\{ \RR e_i \}_{i \in \support (\lambda)} 
= 
\{  \RR (k^{-1} e_j) \}_{j \in \support (\mu)}.
\]
In particular, one can find and fix $(\sigma,(\varepsilon_1,\dots,\varepsilon_n)) \in \mathfrak{S}_n \ltimes \{ \pm 1 \}^n$ satisfying that the equality $ke_i = \varepsilon_i e_{\sigma(i)}$ holds for each $i \in \support (\lambda)$.
Let us define $k' \in W \subset O(n)$ to be the element corresponding to $(\sigma,(\varepsilon_1,\dots,\varepsilon_n)) \in \mathfrak{S}_n \ltimes \{ \pm 1 \}^n$ fixed as above.
Then for $k_0 := k'^{-1} k \in O(n)$, 
we have $k_0 e_i = e_i$ for each $i \in \support (\lambda)$ and $\mathrm{Span} \{ e_i \mid \lambda_i = 0 \}$ is stable by $k_0$.
Hence $k_0.C_\lambda = C_\lambda$, and $k'. C_\lambda = k.C_\lambda = C_\mu$.
This completes the proof. 
\end{proof}

\begin{proof}[Proof of Theorem~\ref{theorem:RnDFW}]
The continuous map $\psi : V_+ \rightarrow O(n) \backslash \Xi$ is surjective because any $W$-orbit in $V$ intersects $V_+$ and $W \subset O(n)$.
The injectivity of $\psi$ follows from 
Proposition~\ref{proposition:OnW} as below:
Let us take any two elements $P,Q \in V_+$ such that $O(n).P = \psi(P) = \psi(Q) = O(n).Q$.
Then by Proposition~\ref{proposition:OnW}, 
$W.P = W.Q$.
Since $V_+$ gives a complete representative of $W$-action on $V$, we have $P = Q$.
This proves that the map $\psi$ is injective.

We only need to show that the continuous bijective map $\psi$ is closed.
One can see that $V_+$ is closed in $S^3((\RR^n)^\ast)$, and hence the inclusion map $V_+ \rightarrow \Xi$ is closed.
Furthermore, the quotient map $\Xi \rightarrow O(n) \backslash \Xi$ is also closed.
In fact, in general, for a compact Hausdorff group $K$ and a continuous $K$-action on a locally-compact Hausdorff space $Y$, the quotient map $Y \rightarrow K \backslash Y$ is known to be closed.
Thus the composition $\psi : V_+ \rightarrow O(n) \backslash \Xi$ is also closed.
\end{proof}

\end{document}